\documentclass[english,11pt,a4paper,leqno]{amsart}
\usepackage{amsmath,amstext, amsthm, amssymb}
\usepackage{latexsym, amsfonts, graphicx, xcolor}
\usepackage{mathrsfs}
\usepackage[latin1]{inputenc} 
\usepackage[T1]{fontenc} 
\usepackage{enumerate}
\usepackage{mathtools}
\usepackage[english]{babel} 
\usepackage{enumitem}
\usepackage{hyperref}

\makeatletter
\def\namedlabel#1#2{\begingroup
	#2%
	\def\@currentlabel{#2}%
	\phantomsection\label{#1}\endgroup
}

\makeatother

\hypersetup{
	colorlinks   = true, 
	urlcolor     = blue, 
	linkcolor    = blue, 
	citecolor   = red 
}

\numberwithin{equation}{section}

\newtheorem{thm}{Theorem}[section]
\newtheorem{prop}[thm]{Proposition}
\newtheorem{lem}[thm]{Lemma}
\newtheorem{coro}[thm]{Corollary} 

\newtheorem{conj}[thm]{Conjecture}
\newtheorem*{prop*}{Proposition}
\newtheorem*{thm*}{Theorem}
\newtheorem{thmletter}{Theorem}

\theoremstyle{theorem}
\newtheorem*{warning*}{Warning}
\newtheorem*{rem*}{Remark}

\theoremstyle{definition} 

\newtheorem{rem}[thm]{Remark}

\newtheorem*{nota*}{Notation}
\newtheorem{Fact}{Fact}

\newcommand{\I}{\mathcal I}
\newcommand{\J}{\mathcal J}
\newcommand{\Q}{\overline{\mathbb Q}}
\newcommand{\C}{\mathbb{C}}
\newcommand{\Z}{\mathbb{Z}}

\newcommand{\OK}{{\mathcal O_{\mathbb K}}}

\newcommand{\X}{\boldsymbol{X}}
\newcommand{\Y}{\boldsymbol{Y}}

\newcommand{\f}{{\boldsymbol f}}
\newcommand{\g}{{\boldsymbol g}}

\newcommand{\bmu}{{\boldsymbol{\mu}}}
\newcommand{\bnu}{{\boldsymbol{\nu}}}
\newcommand{\omeg}{\boldsymbol \omega}

\newcommand{\p}{{\mathfrak{p}}}

\newcommand{\val}{\operatorname{val}_{z}}
\newcommand{\house}[1]{ \overline{\vert#1\vert} }

\title[Independence measures for values of $E$ and $M$-functions]{Algebraic independence measures for values of $E$-functions and 
$M$-functions}

\author{Boris Adamczewski}
\address{
	Univ Lyon, Universit\'e Claude Bernard Lyon 1\\
	CNRS UMR 5208, Institut Camille Jordan \\
	F-69622 Villeurbanne Cedex, France}
\email{Boris.Adamczewski@math.cnrs.fr}

\author{Colin Faverjon}
\address{
	Univ Lyon, Universit\'e Claude Bernard Lyon 1\\
	CNRS UMR 5208, Institut Camille Jordan \\
	F-69622 Villeurbanne Cedex, France}
\email{colin.faverjon@math.cnrs.fr}
\date{\today}

\date{}
\thanks{}
\begin{document}
\begin{abstract}  
In this article, we establish a Liouville-type inequality for polynomials evaluated at the values of arbitrary Siegel $E$-functions at non-zero algebraic points. 
Additionally, we provide a comparable result within the framework of Mahler $M$-functions.
 \end{abstract}

	\bibliographystyle{abbvr}
	\maketitle

\centerline{ \emph{To Yuri V. Nesterenko for his seventy-fifth (\(+\varepsilon\)) birthday.}}
	
\section{Introduction}\label{sec:introduction}

It is a well-known fact that proofs of irrationality, transcendence, linear independence, and algebraic independence can, in general, be quantified, albeit sometimes with  
considerable effort. These quantitative results are commonly referred to as \emph{measures} of irrationality, transcendence, etc. In this paper, we focus primarily on measures of algebraic independence. 
Let $\Q \subset \C$ denote the field of algebraic numbers. 
Given complex numbers $\xi_1,\ldots,\xi_m$ that are algebraically independent over $\Q$, an algebraic independence measure for $\xi_1,\ldots,\xi_m$ is an inequality of the form  
$$
\vert P(\xi_1,\ldots,\xi_m) \vert \geq \Phi(\deg(P), H(P))\,,
$$
valid for any $P \in \Z[X_1,\ldots,X_m]$, and where the function $\Phi$ is a positive-valued function. Here, we let 
$\deg(P)$ denote the (total) degree of $P$, and $H(P)$  the height of $P$, which is defined as the maximum of the moduli of its coefficients. 

The framework of Siegel $E$-functions provides one of the most important result in transcendental number theory: the Siegel-Shidlovskii theorem (cf.\ \cite[3rd Fundamental Theorem]{Sh_Liv}).  
A similar result also holds for Mahler $M$-functions and is known as Nishioka's theorem (cf.\ \cite[Thm.\ 4.2.1]{Ni_Liv}).  
Not surprisingly, both of these results have been quantified.  

\begin{thmletter}
	\label{thm:LGN}
 Let $m\geq 1$ be an integer,  
 $f_1,\ldots,f_m \in \Q[[z]]$ be algebraically independent over $\Q(z)$ and $\alpha \in \Q\setminus\{0\}$. 
 Then, in the two following cases,  there exist two positive real numbers $C_1$ and $C_2$ such that  the inequality 
		$$
	\vert P(f_1(\alpha),\ldots,f_m(\alpha)) \vert \geq C_1H(P)^{-C_2\deg(P)^m}  \,,
	$$
	holds  for all non-zero polynomials  $P\in \mathbb Z[X_1,\ldots,X_m]$.   
\begin{itemize}
	\item[\rm (E)] The functions $f_1,\ldots,f_m$ are $E$-functions related by a linear differential system $Y'(z)=A(z)Y(z)$,  $A \in \mathcal M_m(\Q(z))$, and $\alpha$ is regular with respect to this system.

	\item[\rm (M$_q$)] The functions $f_1,\ldots,f_m$ are $M_q$-functions related by a Mahler system $Y(z^q)=A(z)Y(z)$, $q\in\mathbb N_{\geq 2}$, $A(z) \in {\rm GL}_m(\Q(z))$, $\alpha$ is regular with respect to this system and $\vert \alpha \vert < 1$.	
\end{itemize}
In each case, $C_1$ depends on $\deg(P)$, $f_1,\ldots,f_r$ and $\alpha$, while $C_2$ depends on $f_1,\ldots,f_r$ and $\alpha$ but not on $\deg(P)$. 
\end{thmletter}

A proof of this result was first provided by Lang~\cite{La62} in Case (E) and by Becker~\cite{Bec91} in Case (M$_q$). It has a long history, particularly concerning the constants $C_1$ and $C_2$, which we will briefly review in Section~\ref{sec:history}. We also refer the reader to that section for definitions related to the theorem, such as those of $E$-functions, $M_q$-functions, and regular points.

Despite its generality, Theorem~\ref{thm:LGN} has three important limitations.

\medskip

\begin{enumerate}[label=(\roman*)]

\item \label{Assump_i} The functions $f_1,\ldots,f_m$ are assumed to be algebraically independent over $\Q(z)$.

\medskip

\item\label{Assump_ii} The functions $f_1,\ldots,f_m$ are assumed to satisfy a linear (differential/Mahler) system of size $m$.

\medskip

\item \label{Assump_iii} The point $\alpha$ is assumed to be regular with respect to the corresponding system mentioned in \ref{Assump_ii}.

\end{enumerate}

There are valid reasons for making these three assumptions. According to the Siegel-Shidlovskii and Nishioka theorems, these conditions guarantee that the numbers $f_1(\alpha),\ldots,f_m(\alpha)$ are algebraically independent. In his monograph \cite[Chap.\ 12, \S 4]{Sh_Liv}, Shidlovskii provided several generalizations of Case (E) of Theorem~\ref{thm:LGN}, where \ref{Assump_i} is replaced by some technical conditions that ensure the algebraic independence of $f_1(\alpha),\ldots,f_r(\alpha)$ for some $r<m$, and then a lower bound for $\vert P(f_1(\alpha),\ldots,f_r(\alpha))\vert$ is derived. 
In this paper, however, we take the opposite approach. Our main goal is to show how to remove these three assumptions and explain why doing so is relevant.

Our first main result is a generalization of Theorem~\ref{thm:LGN} in which Assumption~\ref{Assump_i} is removed. The trade-off for this increased generality is that we can no longer rule out the possibility that $P(f_1(\alpha),\ldots,f_m(\alpha))$ vanishes. Consequently, our results take the form of an alternative, much like the Liouville inequality (see, for instance, \cite[p.\ 84--85]{Wa_Liv}). The latter ensures that, given algebraic numbers $\alpha_1,\ldots,\alpha_m$, there exist two positive real numbers 
$C_1$ and $C_2$ such that  the following alternative holds  for all polynomials $P\in\mathbb Z[X_1,\ldots,X_m]$: 
\begin{equation}\label{eq:liouville}
	\text{either } \;  P(\alpha_1,\ldots,\alpha_m)=0 \quad
\text{ or } \quad \vert P(\alpha_1,\ldots, \alpha_m) \vert \geq C_1H(P)^{-C_2} \,.
	\end{equation}
	Furthermore, $C_1$ depends on  $\alpha_1,\ldots,\alpha_m$  and 
	$\deg(P)$, while $C_2$ depends only on   $\alpha_1,\ldots,\alpha_m$.

\begin{thm}
		\label{thm:measure_system}  Let $m\geq 1$ be an integer, $f_1,\ldots,f_m \in \Q[[z]]$ and $\alpha \in \Q\setminus\{0\}$. Set $t:={\rm tr.deg}_{\Q(z)}(f_1,\ldots,f_m)$. 
		Then, in the two following cases,  there exist two positive real numbers $C_1$ and $C_2$ such that  the 
		 following alternative holds  for all polynomials $P\in\mathbb Z[X_1,\ldots,X_m]$:   
	\begin{eqnarray*}
	\text{ either } & P(f_1(\alpha),\ldots,f_m(\alpha))=0 \\
\text{ or } & \vert P(f_1(\alpha),\ldots,f_m(\alpha)) \vert \geq C_1H(P)^{-C_2\deg(P)^t} \,.
	\end{eqnarray*}	  
		\begin{itemize}
	\item[\rm (E)] The functions $f_1,\ldots,f_m$ are $E$-functions related by a linear differential system $Y'(z)=A(z)Y(z)$,  $A \in \mathcal M_m(\Q(z))$, and $\alpha$ is regular with respect to this system.

	\item[\rm (M$_q$)] The functions $f_1,\ldots,f_m$ are $M_q$-functions related by a Mahler system $Y(z^q)=A(z)Y(z)$, $q\in\mathbb N_{\geq 2}$, $A(z) \in {\rm GL}_m(\Q(z))$, $\alpha$ is regular with respect to this system and $\vert \alpha \vert < 1$.	
\end{itemize}
		In each case, $C_1$ depends on $\deg(P)$, $f_1,\ldots,f_r$ and $\alpha$, while $C_2$ depends on $f_1,\ldots,f_r$ and $\alpha$ but not on $\deg(P)$. 
\end{thm}

\begin{rem}
Here are a few comments regarding Theorem~\ref{thm:measure_system}. 

\medskip

\begin{itemize}

\item When $t=m$, it has already been noted that $f_1(\alpha),\ldots,f_m(\alpha)$ are algebraically independent, and thus $P(f_1(\alpha),\ldots,f_m(\alpha))\not=0$.  
In this case, Theorem~\ref{thm:LGN} is recovered as a special case. Moreover, if $r<m$ is such that $f_1(\alpha),\ldots,f_r(\alpha)$ are algebraically independent, then 
it follows that the inequality $\vert P(f_1(\alpha),\ldots,f_r(\alpha)) \vert \geq C_1H(P)^{-C_2\deg(P)^t}$ holds  for all polynomials $P\in\mathbb Z[X_1,\ldots,X_r]$.

\medskip

\item Assumption \ref{Assump_i} poses a significant limitation when dealing with the values of general $E$- or $M$-functions. For instance, consider a case where we are interested in a set of algebraically independent $E$-functions, say $f_1,\ldots,f_r$,  While it is always possible to construct a linear differential system that relates these functions, this system will typically involve additional functions, resulting in a larger size, say $m>r$.  Consequently, new $E$-functions, denoted $g_1,\ldots,g_{m-r}$, may emerge. However, there is no inherent reason to expect that the combined set of functions $f_1,\ldots,f_r,g_1,\ldots,g_{m-r}$ will remain algebraically independent. 
As a result, Theorem~\ref{thm:LGN} may no longer be applicable in such cases.

\medskip

\item In Case (E), our proof follows the classical approach developed by Siegel and Shidlovskii, which is also the one used by Lang in \cite{La62}. In Case (M$_q$), our proof relies on the classical framework introduced by Nesterenko and further developed by Becker \cite{Bec91} and Nishioka \cite{Ni91,Ni_Liv}. It should be noted that no new tools are introduced here, and Theorem~\ref{thm:measure_system} could have been established much earlier using the existing techniques.

\medskip

\item Bounds for the values of $C_1$ and $C_2$, as well as  their effectivity, are discussed in Sections~\ref{sec:E_func_measures} and \ref{sec:M_func}.

\medskip

\item Let us conclude with a final remark concerning  Case (E) of Theorem~\ref{thm:measure_system}.  When $P$ is a linear form and $P(f_1(\alpha),\ldots,f_m(\alpha))\not=0$, Theorem~\ref{thm:measure_system} provides a measure of linear independence. Until recently, such measures were known to exist--thanks to Shidlovskii--only for values of linearly independent 
$E$-functions with coefficients in $\mathbb K$, evaluated at points in $\mathbb K$, in the two cases where $\mathbb K$ is either the field of rational numbers or an imaginary quadratic  field. However, Fischler and Rivoal~\cite{FR23} have recently succeeded in lifting this restriction by employing a Galois-theoretic argument that relies on Beuker's Lifting Theorem~\cite{Be06}. In contrast, we obtain a linear independence measure valid for any number field $\mathbb K$ by utilizing techniques that have been available since the late 1950s.

\end{itemize}
\end{rem}

To complete our program, we now turn to removing Assumptions \ref{Assump_ii} and \ref{Assump_iii} from Theorem~\ref{thm:LGN}, which is accomplished with the following result.

\begin{thm}
	\label{thm:measure_sans_system} Let $r\geq 1$ be an integer,  $f_1,\ldots,f_r \in \Q[[z]]$ and $\alpha \in \Q\setminus\{0\}$.  Then, in the two following cases,  there exist two positive real numbers $C_1$ and $C_2$ such that  the 
		 following alternative holds for all polynomial $P\in\mathbb Z[X_1,\ldots,X_m]$:   
	\begin{eqnarray*}	
	\text{ either } & P(f_1(\alpha),\ldots,f_m(\alpha))=0 \\
\text{ or } & \vert P(f_1(\alpha),\ldots,f_m(\alpha)) \vert \geq C_1H(P)^{-C_2\deg(P)^\tau} \,.
	\end{eqnarray*}  
	\begin{itemize}
		\item[\rm (E)] The functions  $f_1,\ldots,f_r$ are $E$-functions and 
		$$\tau = {\rm tr.deg}_{\Q(z)}(f_i^{(\ell)} \,:\, 1\leq i\leq r, \; \ell \geq 0)\,.$$
		
		\item[\rm (M$_q$)] The functions $f_1,\ldots,f_r$ are $M_q$-functions for some $q\in \mathbb N_{\geq 2}$, which are well-defined at $\alpha$, $\vert \alpha \vert < 1$, and 
		$$\tau = {\rm tr.deg}_{\Q(z)}(f_i(z^{q^\ell}) \,:\, 1\leq i\leq r,\; \ell \geq 0)\,.$$	
	\end{itemize}
		In each case, $C_1$ depends on $\deg(P)$, $f_1,\ldots,f_r$ and $\alpha$, while $C_2$ depends on $f_1,\ldots,f_r$ and $\alpha$ but not on $\deg(P)$. 
	\end{thm}

\begin{rem}
 If for every $i$, $1\leq i\leq r$, we let  $m_i$ denote the order of the minimal linear differential/Mahler equation satisfied by $f_i$, then we have the inequality 
\begin{equation*}
\tau \leq m_1+\cdots +m_r \,.
\end{equation*} 
\end{rem}

The key to deriving Theorem~\ref{thm:measure_sans_system} from Theorem~\ref{thm:measure_system} lies in the existence of certain \emph{good operators} associated with arbitrary $E$- or $M$-functions. Specifically, all we require is that for any $\alpha$ and any $f$, there exists a linear differential/Mahler operator that is regular at $\alpha$. 
When $f$ is an $E$-function, the existence of such a differential operator is guaranteed by a result of André \cite{An00}, which states that $f$ is annihilated by a differential operator whose singularities are confined to $\{0, \infty\}$. For $M$-functions, the existence of such a Mahler operator was recently established by the authors \cite{AF24_EM}.

We conclude this introduction by presenting a consequence of Theorem~\ref{thm:measure_sans_system}, which highlights the relevance of our approach.
Recall that a Liouville number is an irrational real number $\xi$ for which, for all $w \geq 1$, there exists a rational number $p/q \in \mathbb{Q}$ such that 
$$
\left\vert \xi - \frac{p}{q}\right\vert \leq \frac{1}{q^w}\,\cdot
$$
Let $\mathbf{E}$ denote the set of values of $E$-functions at non-zero algebraic points, and similarly, let $\mathbf{M}$ denote the set of values of $M$-functions 
at non-zero algebraic points (where the evaluation is well-defined). The question of whether either of these two sets can contain Liouville numbers has been an old problem, 
with roots in the pioneering works of Siegel and Mahler (see the discussion in Section~\ref{sec:history}). 
For the set $\mathbf{E}$, this problem was fully resolved only very recently by Fischler and Rivoal~\cite{FR23}, who employed a desingularization procedure due to Beukers, along with a Galois-theoretic argument. Both methods are based on Beuker's Lifting Theorem~\cite{Be06}. 
For the set $\mathbf{M}$, the result was announced by Zorin~\cite{Zo13} in an unpublished preprint. 
As a direct consequence of Theorem~\ref{thm:measure_sans_system}, we obtain the following result, which provides a new proof of Fischler and Rivoal's result, one that does not rely on Beukers' Lifting Theorem, and also confirms Zorin's claim.

\begin{coro}{\, }
	\label{coro:Lnumbers}
No $\xi\in \mathbf E\cup \mathbf M$ is a Liouville number. 
\end{coro}

In fact, from Theorem \ref{thm:measure_sans_system}, we deduce a more general result: no $\xi \in \mathbf{E} \cup \mathbf{M}$ is a $U$-number in Mahler's classification (cf.\ Section~\ref{sec:def} for more details).

\subsection{Organization of the Paper} The paper is structured as follows. In Section~\ref{sec:def}, we recall the main  definitions related to $E$- and $M$-functions, along with Mahler's classification. In Section~\ref{sec:E_funct},  we prove Case (E) of Theorem~\ref{thm:measure_system}, while Section~\ref{sec:M_func} addresses Case (M$_q$). Theorem~\ref{thm:measure_sans_system} is established in Section~\ref{sec:remove_sing}. Finally, Section~\ref{sec:history} presents a brief overview of previous results relevant to this work.

\section{Main definitions}\label{sec:def}

This section aims to recall the main definitions related to $E$- and $M$-functions, as well as Mahler's classification.

\subsection{$E$-functions} 

An $E$-function is a power series of the form $$f(z)=\sum_{n=0}^\infty \frac{a_n}{n!}z^n$$ that satisfies a linear differential equation with coefficients in $\Q[z]$ 
and whose arithmetic growth of coefficients is constrained by the following two conditions:
there exist $C>0$ and a sequence of integers $d_n\geq 1$ such that for all $\sigma \in  {\rm Gal}(\Q/\mathbb Q)$ and all integers $n\geq 1$, 
we have $\vert \sigma(a_n)\vert \leq C^n$, $d_n\leq C^n$, and $d_na_i$ is an algebraic integer for every $i$, $1\leq i\leq n$. 

Typical examples of $E$-functions include $e^z$, $\cos z$, $\sin z$, the Bessel function, and, more generally, the hypergeometric functions 
of the following form:
$$
{}_pF_{q}
\left[
\begin{matrix}
a_1, \ldots, a_p
\\
b_1, \ldots, b_q
\end{matrix}
\,; \lambda z^{q-p} \right] := \sum_{n=0}^\infty \frac{(a_1)_n\cdots (a_p)_n}{(b_1)_n\cdots (b_q)_n} \lambda^nz^{n(q-p)}\,,
$$
where $0\leq p <q$, $a_1,\ldots,a_p\in\mathbb Q$, $b_1,\ldots,b_q\in\mathbb Q\setminus \mathbb Z_{\leq 0}$, $\lambda \in \Q$.  Here,  
$(\cdot)_n$ denotes the Pochhammer symbol.

When studying $E$-functions, it is often more convenient to consider linear differential systems  of order $1$, i.e., systems of the form 
\begin{equation}\label{eq:DiffSystem}
{\bf Y}'(z)
=
A(z){\bf Y}(z)\,
,\quad\quad A(z) \in {\mathcal M}_m(\Q(z))\, ,
\end{equation} 
where the system has a vector solution whose coordinates are $E$-functions. Conversely, we recall that for any arbitrary $E$-functions $f_1, \dots, f_r$, there exists 
a vector solution to a system of the form \eqref{eq:DiffSystem} with $m \geq r$, where $f_1, \dots, f_r$ appear as some of its coordinates. 
Additionally, we also recall that a point $\alpha \in \mathbb{C}$ is said to be \emph{regular} with respect to the system \eqref{eq:DiffSystem} if $A(z)$ is well-defined at $\alpha$, i.e., if $\alpha$ is not a pole of any of the coordinates of $A(z)$. Otherwise, 
$\alpha$ is said to be \emph{singular}.

\subsection{$M$-functions}

Let $q \geq 2$ be an integer. An $M_q$-function is a power series $f(z) \in \Q[[z]]$ that satisfies an equation of the form
$$
a_0(z)f(z)+a_1(z)f(z^q)+\cdots + a_m(z)f(z^{q^m})=0\,,
$$
where $a_0, \dots, a_m \in \Q[z]$ and $a_0 a_m \neq 0$. We refer to a power series as an $M$-function when it is an $M_q$-function for some $q \geq 2$, where the value of $q$ 
is not necessarily specified.

Typical examples of $M$-functions include:  
$$\sum_{n=0}^{\infty} z^{3^n} \,, \quad \prod_{n=0}^\infty \frac{1}{1-z^{5^n}} \,, \quad \sum_{n=0}^\infty \lfloor \log_2 n \rfloor z^n\,,$$
and, most importantly, the generating series of automatic sequences (see, for  instance, \cite{ABS}, for more examples). 

When studying $M_q$-functions, it is often more convenient to consider linear Mahler systems of order $1$, i.e., systems of the form 
\begin{equation}\label{eq:MSystem}
{\bf Y}(z^q)
=
A(z){\bf Y}(z)\,
,\quad\quad A(z) \in {\rm GL}_m(\Q(z))\, ,
\end{equation} 
 which have a vector solution whose coordinates are $M_q$-functions. Conversely, we recall that for any arbitrary $M_q$-functions $f_1, \dots, f_r$, 
 there exists a vector solution to a system of the form \eqref{eq:MSystem} with $m \geq r$, where $f_1, \dots, f_r$ appear as some of its coordinates. 
We also recall that a point $\alpha \in \mathbb{C}$ is said to be \emph{regular} with respect to the system \eqref{eq:MSystem} if $A(z)$ is well-defined and invertible at $\alpha^{q^n}$ for every non-negative integer $n$, and \emph{singular} otherwise.

Unlike $E$-functions, an $M$-function is not an entire function unless it is a polynomial. An $M$-function is always analytic in some neighborhood of the origin and can be meromorphically continued within the open unit disk. Furthermore, the unit circle is a natural boundary, unless the function is rational. This is why the additional  assumption 
$\vert \alpha \vert < 1$ is made in all our results concerning $M$-functions.

\subsection{Mahler's classification}

When a number $\xi$ is proven to be transcendental, it is natural to ask whether it could be a Liouville number or, more generally, 
to classify it within Mahler's framework, which we now outline.

Given a complex number $\xi$ and a positive integer $d$, let $w_d(\xi)$ denote the supremum of the real numbers $w$ such that the inequality 
$$
0<\vert P(\xi) \vert <\frac{1}{H(P)^w} \,,
$$
holds for infinitely many polynomials $P \in \Z[X]$ with degree at most $d$. 
It follows that $\xi$ is a Liouville number if and only if $w_1(\xi)=+\infty$. 
We then define
$$
w(\xi):=\limsup \frac{w_d(\xi)}{d}\,\cdot 
$$
The complex numbers can be classified into the following four classes: 
\begin{itemize}
	\item $\xi$ is an $A$-number if $w=0$.
	\item $\xi$ is an $S$-number if $0 < w < \infty$.
	\item $\xi$ is a $T$-number if $w=\infty$ and $w_d(\xi)<\infty$ for all $d\geq 1$.
	\item $\xi$ is a $U$-number if $w_d(\xi)=\infty$ for some $d\geq 1$. 
\end{itemize}
$A$-numbers are precisely the algebraic numbers, while $U$-numbers generalize Liouville numbers. Almost all numbers (in the sense of Lebesgue measure) are $S$-numbers. Distinguishing between $S$- and $T$-numbers is notoriously difficult in transcendental number theory.

Given this classification, the following general and basic heuristic emerges once the transcendence of a complex number, say $\xi$, has been proven:  
$\xi$ is likely to be an $S$-number, and thus not a Liouville number nor a 
$U$-number. Of course, this conjecture is only made if there is no evidence to the contrary, which is indeed the case for elements of the sets 
$\bf E$ and $\bf M$.

\begin{conj}\label{conj}
All transcendental elements in ${\bf E}\cup {\bf M}$ are $S$-numbers. 
\end{conj}

In light of the remark under Corollary \ref{coro:Lnumbers}, the last step to establish this conjecture is to prove that no $\xi \in {\bf E}\cup {\bf M}$ is a $T$-number.

\section{Proof of Case (E) of Theorem \ref{thm:measure_system}}\label{sec:E_funct}

Our proof follows the classical method introduced by Siegel and extended by Shidlovskii. 
For more details, we refer the reader to \cite{FN} and \cite{Sh_Liv}. 
We assume that the reader is familiar with the proof of the Siegel-Shidlovskii theorem and will not provide the full details of all the computations. 
Instead, we focus on the new argument that we introduce.

\medskip

Let us begin by introducing some notations that we will use throughout this section. Let $\f := (f_1(z), \ldots, f_m(z))$ and $d := \deg(P)$, where $P$ is the polynomial in Theorem~\ref{thm:measure_system}. 
Let $\mathbb K$ be a number field containing the coefficients of  $f_1(z),\ldots,f_m(z)$ as well as the point $\alpha$. Let $h:=[\mathbb K:\mathbb Q]$ denote its degree 
and  $\OK$ its ring of integers.  Let $T(z)$ denote 
the least common denominator of the entries of the matrix $A(z)$. 
Multiplying $T(z)$ by a scalar if necessary, we can assume that all the coefficients of $T(z)A(z)$ belong to 
$\OK[z]$.  
When working with $E$-functions, it is convenient to consider the house as a suitable notion of height for an algebraic integer. The house of an algebraic number is defined as the maximum modulus of its Galois conjugates. To facilitate the transition from this note to the existing literature on $E$-functions, we adopt this notion of height rather than the absolute Weil height.  
Throughout this section, the height $H(Q)$ of a polynomial $Q \in \OK[X_1,\ldots,X_m]$ is defined as the maximum of the houses of its coefficients. 
When $Q$ has coefficients in $\Z$, this corresponds to the classical notion of height used in the introduction. 
Along the proof, we will introduce various constants, all of which are positive real numbers:
 $\delta,\delta_i,\,\eta_i$ 
do not depend on $d$, while  $p,q,r,s,u,v,w,n_i,\rho_i$ may depend on  $d$.

\medskip

Without loss of generality, we can assume that $\alpha = 1$ by replacing the functions  $f_1(z), \ldots, f_m(z)$ with $f_1(\alpha z), \ldots, f_m(\alpha z)$ if necessary. 
From now on, we assume that 
\begin{equation}\label{eq:pnot0} 
P(f_1(1), \ldots, f_m(1)) \neq 0, 
\end{equation} and we aim to establish the required lower bound for $|P(f_1(1), \ldots, f_m(1))|$. We note that when $f_1(z), \ldots, f_m(z)$ are all polynomials, the result follows from the Liouville inequality (see \eqref{eq:liouville}). Therefore, we assume that at least one of them is not a polynomial. Since an $E$-function is either a polynomial or transcendental, it follows that $t \geq 1$.

\subsection{Preliminary results on dimensions of some vector spaces}\label{sec:Efunc_dimension}

Given a field $\bf k$ and an integer $\delta\geq 0$, we let ${\bf k}[\X]_{\leq \delta}$ denote the ${\bf k}$-vector space of polynomials 
with total degree at most $\delta$ in $\X=(X_1,\ldots,X_m)$. Given an ideal $\I \subset {\bf k}[\X]$, 
we set $\I_{\leq \delta}:=\I\cap{\bf k}[\X]_{\leq \delta}$ and if $\mathcal A={\bf k}[\X]/\mathcal I$, we set $\mathcal A_{\leq \delta}:={\bf k}[\X]_{\leq \delta}/\I_{\leq \delta}$. 
For every tuple of non-negative integers $\bmu=(\mu_1,\ldots,\mu_m)$, 
we set 
$$
\X^{\bmu}=X_1^{\mu_1}\cdots X_m^{\mu_m} \; \mbox{ and } \; \f^{\bmu}=f_1^{\mu_1}\cdots f_m^{\mu_m} \,.
$$
Given a prime ideal $\I$, we let ${\rm ht}(\mathcal I)$ denote its height, that is the maximum of the integers $h$ such that there exist 
primes ideals $\mathfrak p_0,\ldots,\mathfrak p_h$ satisfying
$$
\mathfrak p_0 \subsetneq \mathfrak p_1 \subsetneq \cdots \subsetneq \mathfrak p_h = \I\,.
$$
The height of a possibly non-prime ideal $\I$ is defined as 
$$
{\rm ht}(\I):=\min\{h(\mathfrak p) : \mathfrak p \mbox{ is prime and } \I\subseteq \mathfrak p\} \,.
$$
Let
$$
\I_z(\f):=\{Q(z,\X) \in \mathbb K(z)[\X] \, :\,  Q(z,\f(z))=0\}
$$
	denote the ideal of algebraic relations over $\mathbb K(z)$ between $f_1(z),\ldots,f_m(z)$.  
	Then $\I_z(\f)$ is a prime ideal whose height is $m-t$ (see \cite[Thm.\ 20,\,p.\,193]{ZS}). Let 
	$$
	\mathcal I_1(\f):=\{Q(\X)\in \mathbb K[\X] \, : \, Q(\f(1))=0\}\,
	$$
	denote the ideal of algebraic relations over $\mathbb K$ between $f_1(1),\ldots,f_m(1)$.
	It is a prime ideal which contains the ideal 
	\begin{equation}\label{eq:ev}
	{\rm ev}_1(\I_z(\f)) := \{Q(1,\X) \in \mathbb K[\X] \, :\,  Q(z,\X)\in \I_z(\f)\cap \mathbb K[z][\X]\} \,,
	\end{equation}
	 form by the valid specializations at $z=1$ of elements of $\I_z({\f})$\footnote{According to Beukers' Lifting Theorem \cite[Thm.\ 1.3]{Be06}, 
	 we actually know that 
	${\rm ev}_1(\I_z(\f))=\mathcal I_1(\f)$, but we want to emphasize that this result is not needed for our proof.}. 
	The height of $\I_1(\f)$ is greater than or equal to $m-t$\footnote{The Siegel-Shidlovskii Theorem implies that the height of $\I(\f(1))$ is actually 
	equal to $m-t$, but, again, we want to emphasize that we do not need to use this stronger fact.}. 
	We also set 
	\begin{equation}\label{eq:defJ}
	\J:=\I_1(\f)+(P) \subset \mathbb K[\X]\,.
	\end{equation}
	 The ideal $\J$ is not necessarily prime, but since $\I_1(\f)$ is prime and 
	(according to \eqref{eq:pnot0}) 
	$P\not\in\I_1(\f)$,  the height of $\J$ is at least $m-t+1$.
	
	\medskip Let $\delta\geq 1$ be an integer. Set 
	\begin{eqnarray}
	\nonumber p:=\dim_{ \mathbb K}  \mathbb K[\X]_{\leq \delta d}, &  q:=\dim_{ \mathbb K(z)} \I_z(\f)_{\leq \delta d}, & r:=\dim_{ \mathbb K} \I_1(\f)_{\leq \delta d}\\
	s:=\dim_{ \mathbb K} \J_{\leq \delta d},& u:=s-r,&v:=p-s, \\ 
	\nonumber w:=p-q\,.
	\end{eqnarray}

The aim of this section is to prove the following result.

\begin{lem}\label{lem:7-8}
There exist a positive integer $\delta_1$ and a positive real number $\eta_1$, such that for every $\delta\geq \delta_1$, we have 
\begin{equation*}
vh<w\,,   \mbox{ and }\, u\leq  \eta_1 \delta^{t} d^{t}\,.
\end{equation*}
\end{lem}

\begin{proof}
By  the Hilbert-Serre Theorem (see, for instance,  \cite[p.\,231]{FN}),  if $\delta$ is large enough, say $\delta\geq \gamma_1$, 
$w$ is a polynomial of degree $t$ in $\delta d$ and hence there  
 exist two positive real numbers $\mu_1,\mu_2$ such that
\begin{equation}
\label{eq:majo_w}
 \mu_1\delta^{t} d^{t} \leq w\leq \mu_2\delta^{t}d^{t} \,.
\end{equation}

Set $\mathcal A:=\mathbb K[X_1,\ldots,X_m]/\I_1(\f)$. 
Then $\dim_{\mathbb K}\mathcal A_{\leq \delta d}=p-r$. 
Using again the Hilbert-Serre Theorem as in \cite[p.\,231]{FN}, we obtain the existence of a positive integer $\gamma_2$ and a polynomial $\varphi(x)$ 
of degree $t_0\leq t$ such that, for every $\delta \geq \gamma_2-1$, $p-r=\varphi(\delta d)$. 
Let $\overline{P}$ denote the image of $P$ in $\mathcal A$. 
Since 
\begin{equation*}
\mathbb K[X_1,\ldots,X_m]/\J  \cong  \mathcal A/\overline{P}\mathcal A\,,
\end{equation*}
we have $\dim_{\mathbb K}(\overline{P}\mathcal A)_{\leq \delta d}=s-r=u$. The map $Q\to Q\overline{P}$, from $\mathcal A_{\leq (\delta-1)d}$ 
to $(\overline{P}\mathcal A)_{\leq \delta d}$ is injective since $\mathcal A$ is a domain. It is also surjective, by definition. 
Hence when $\delta\geq \gamma_2$, 
we have $u= \varphi((\delta -1)d)$. It follows that 
	$$
	v=p-s=p-r-u= \varphi(\delta d)-\varphi((\delta-1)d)\,.
	$$
	The right-hand side is a polynomial with degree $t_0$ in $d$ and $t_0-1$ in $\delta$. 
	 There thus exists a positive real number $\mu_3$ such that
 \begin{equation}
\label{eq:majo_v}
 v \leq \mu_3 \delta^{t_0-1} d^{t_0}\,,
\end{equation}
as soon as $\delta \geq \gamma_2$.

Since $t_0-1<t$, we deduce from \eqref{eq:majo_w} and \eqref{eq:majo_v} that 
$w>vh$, as soon as $\delta$ is sufficiently large, say $\delta\geq \gamma_3\geq \max\{\gamma_1,\gamma_2\}$.  
On the other hand, we have $u=r-s\leq p-q = w$. According to \eqref{eq:majo_w}, we obtain the expected result by setting $\delta_1:=\gamma_3$ and $\eta_1:=\mu_2$. 
\end{proof}

Until the end of the proof, we fix an integer $\delta\geq \delta_1$, so that the conclusion of Lemma~\ref{lem:7-8} holds.

\subsection{A first basis of linear forms}

Recall that $p=\dim \mathbb K[\X]_{\leq \delta d}$. Let $\Y=(Y_1,\ldots,Y_p)$ be a $p$-tuple of indeterminates and let us consider a bijective linear map
\begin{equation}\label{eq:psi}
\psi : \mathbb K(z)[\X]_{\leq \delta d} \to \mathbb K(z)Y_1+\cdots +\mathbb K(z)Y_p
\end{equation}
sending each monomial $\X^\bmu$ to one of the variables $Y_1,\ldots,Y_p$. 
For every $i$, $1\leq i\leq p$, we set  $g_i(z):=\f(z)^{\bmu}$, where $\bmu$ is defined by $\psi^{-1}(Y_i)=\X^{\bmu}$.  
By definition, the power series $g_1,\ldots,g_p$ span a $\mathbb K(z)$-vector space of dimension $w$. Choosing another bijection $\psi$ if necessary, we can assume 
$g_1,\ldots,g_w$, are linearly independent over $\mathbb K(z)$. 
In this section, our aim is to build a basis of the $\mathbb K$-vector space $\mathbb K Y_1+\cdots +\mathbb K Y_p$, 
which is formed by two families of linear forms with size, respectively, $q$ and $p-q=w$.

\medskip

\begin{itemize}
 \item[{\rm (a)}] The linear forms $L_1(\Y),\ldots,L_q(\Y)$ are obtained by specialization at $z=1$ of some linear relations over $\mathbb K(z)$ between $g_1,\ldots,g_p$. They \emph{vanish} when specializing each $Y_i$ at $g_i(1)$, $1\leq i\leq p$.
 
 \medskip
 
 \item[{\rm (b)}] The linear forms $M_1(\Y),\ldots,M_w(\Y)$ are obtained by derivation, and then specialization at $z=1$, 
 of some Padé approximant associated with  $g_1,\ldots,g_w$. They take \emph{small values} when specializing each $Y_i$ at $g_i(1)$, $1\leq i\leq p$.
\end{itemize}

\begin{rem}
In standard proofs of the Siegel-Shidlovskii theorem, it is customary to reduce to the case where the functions $g_i$ are linearly independent. However, this reduction requires modifying the underlying differential system, which incurs a cost in terms of quantification. In contrast, the approach developed here avoids this reduction entirely. Instead, it leverages  the potential linear relations to construct the linear forms (a) that vanish when specializing each $Y_i$ at $g_i(1)$ for $1\leq i\leq p$. Notably, we apply Siegel's method using 
$p$ variables without any reduction in number.
\end{rem}

\subsubsection{Linear forms coming from the functional relations}\label{sec:formes_fonc} In order to build $L_1,\ldots,L_q$, 
we could just take any basis of $\I_z(\f)_{\leq \delta d}$, specialize at $z=1$ and take the image by $\psi$. However, we have to be careful: we actually want 
the height of these linear forms to be  bounded independently of $d$ and $\delta$. 

We proceed as follows.
We consider a monomial ordering on $X_1,\ldots,X_m$ and choose some polynomials 
$Q_1(z,\X),\ldots,Q_k(z,\X) \in \OK[z,\X]$ such that $Q_1(1,\X),\ldots,Q_k(1,\X)$ form a 
Gr\"obner basis\footnote{See \cite{BWK93} for more details on Gr\"obner basis.} of
the ideal ${\rm ev}_1(\I_z(\f))$ defined in \eqref{eq:ev}. 
Let us assume that $\delta$ is large enough, say $\delta\geq \delta_2\geq \delta_1$, so that 
 $\delta d$ is larger than the total degree in $\X$ of each of the polynomials $Q_i$ (note that the later do not depend on $d$, so that $\delta_2$ does not depend on $d$). 
Let $\ell_1(z,\Y),\ldots,\ell_b(z,\Y) \in \OK[z,\Y]$ be an enumeration of all the linear forms of the form 
$$
\psi\left(\X^{\bnu}Q_{i}(z,\X)\right) \,,
$$
where $1\leq i \leq k$ and $\X^{\bnu}$ is a monomial such that $\X^{\bnu}Q_i(z,\X) \in \mathbb K(z)[\X]_{\leq \delta d}$.

\begin{lem}
	\label{lem:indep_Li}
	The rank of $\ell_1(1,\Y),\ldots,\ell_{b}(1,\Y)$ is equal to $q$.
\end{lem}

\begin{proof}
	Set $V:={\rm ev}_1(\I_z(\f))\ \cap \mathbb K[\X]_{\leq \delta d}$. We have $\dim_{\mathbb K}(V)=q$ 
	(see, for example, \cite[Lem.\ 3.1]{Be06}). Since $\psi$ is injective, we only have to prove that the polynomials
	$\X^{\bmu}Q_{i}(1,\X)$ span $V$. Let $d_1,\ldots,d_k$ denote, respectively, the degrees of $Q_1(1,\X),\ldots,Q_k(1,\X)$. Set 
	\begin{multline*}
	W:=\{A_1(\X)Q_1(1,\X)+\cdots+A_k(\X)Q_k(1,\X)\, : 
	\\  A_i \in \mathbb K[\X]_{\leq d\delta - d_i},\,1\leq i \leq k \}\subset V
	\end{multline*}
and assume by contradiction that $W\subsetneq V$. Let $Q(\X) \in V \setminus W$ having the least dominant monomial 
possible among the elements of $V \setminus W$. 
Since $Q_1(1,\X),\ldots,Q_{k}(1,\X)$ is a Gr\"obner basis, there exists an integer $i$ and a polynomial $A(\X)$ such that the dominant monomial of 
$Q(\X) - A(\X)Q_i(1,\X)$ is strictly less than the one of $Q(\X)$. Necessarily, $\deg A(\X) \leq \delta d - d_i$. 
By minimality we have $Q(\X) - A(\X)Q_i(1,\X) \in W$. Since $A(\X)Q_i(1,\X) \in W$ we have $Q(\X) \in W$, a contradiction.
\end{proof}

It follows from Lemma~\ref{lem:indep_Li} that the rank of $\ell_1(1,\Y),\ldots,\ell_{b}(1,\Y)$ equals $q$. 
Without loss of generality, we can assume that $\ell_1(1,\Y),\ldots,\ell_{q}(1,\Y)$ are linearly independent over $\mathbb K$. 
In particular $\ell_1(z,\Y),\ldots,\ell_{q}(z,\Y)$ are linearly independent over $\mathbb K(z)$. We then set
\begin{equation}\label{eq:Ll}
L_1(\Y):=\ell_1(1,\Y),\ldots, L_q(\Y):=\ell_q(1,\Y)\,,
\end{equation}
which ends the construction of the linear forms corresponding to (a). 

\subsubsection{Linear forms coming from Siegel's method}\label{sec:padé} 

This section is devoted to the proof of the following result.   

\begin{prop}
	\label{prop:formes_Padé}
	Let $\varepsilon$ be a positive real number. Then there exist positive integers $n_1,\rho_1,\rho_2$ 
	(depending on $d$ and $\varepsilon$) such that, for every positive integer $n\geq n_1$, there exist linear forms 
	$M_i(\Y) \in \OK[Y]$, $1\leq i\leq w$, satisfying the three following conditions. 
	
	\medskip
	
	\begin{itemize}
		\item[\rm (a)] $\vert M_i(g_1(1),\ldots,g_p(1)) \vert \leq \rho_1 n^{-(w-1-\varepsilon)n}$, $1\leq i \leq w$.
		\medskip
		
		\item[\rm (b)] $H(M_i)\leq \rho_2 n^{(1+\varepsilon)n}$, $1\leq i \leq w$.	
		
		\medskip
		
		\item[\rm (c)] The $p$ linear forms $L_1(\Y),\ldots,L_q(\Y),M_1(\Y),\ldots,M_w(\Y)$ are linearly independent over $\mathbb K$. 
	\end{itemize}
\end{prop}

	The proof of Proposition~\ref{prop:formes_Padé} is a slight modification of Siegel's original argument (see, for instance, \cite[p.\,218--219]{FN}). 
	First we choose a positive real number $\varepsilon_0<1/2$, which will be chosen small enough later in the proof, and an integer $n$, 
	which will be chosen large enough later in the proof. 
	Siegel's lemma then classically implies the existence of a linear form
	$$
	R_0(z,\Y)=P_1(z)Y_1+\cdots+P_w(z)Y_w \in \OK[z,\Y]
	$$
	and a positive real number $c_1$, which does not depend on $n$, such that 
	\begin{eqnarray}\label{eq:slem}
	\nonumber \deg P_i&\leq& n \\
	H(P_i)&\leq& c_1 n^{(1+\varepsilon_0)n}\\
	\nonumber \val(R_0(z,\g(z)))&\geq& w(n+1)-\varepsilon_0 n -1\,,
	\end{eqnarray}
	where $\g:=(g_1,\ldots,g_p)$ and where we let $\val$ denote the usual valuation on $\mathbb K((z))$.  
	Given any linear form $L(z,\Y)\in \OK[z,\Y]$, 
	we can rewrite $T(z)\partial L(z,\g(z))$ as a linear form in $g_1,\ldots,g_p$, say 
	$$
	T(z)\partial L(z,\g(z))=:\sum_{i=1}^p B_i(z)g_i(z) \,,
	$$ 
	by using the differential system linking $g_1(z),\ldots,g_p(z)$. 
	Setting $$\Theta(L(z,\Y)) := \sum_{i=1}^p B_i(z) Y_i \in \OK[z,\Y]\, ,$$ we define a differential operator $\Theta$ in $\OK[z,\Y]$.   
	Now, for every $k\geq 0$, we set $$R_{k+1}(z,\Y):=\Theta( R_k(z,\Y))\,.$$

	Using Shidlowskii's lemma, we deduce the following result. 
	
	\begin{Fact}\label{fact1}
	There exists $n_0$ (which depends on $\varepsilon_0$, $p$ and $g_1(z),\ldots,g_p(z)$), such that if $n\geq n_0$, the linear forms $R_0(z,\g(z)),\ldots,R_{w-1}(z,\g(z))$ 
	are linearly independent over $\mathbb K(z)$. 
	\end{Fact}
	
	\begin{proof}
	We recall that the power series $g_1(z),\ldots,g_w(z)$ are linearly independent 
	over $\mathbb K(z)$ and form a vector solution to a linear differential system (which can be easily deduced from the original differential system 
	linking $g_1(z),\ldots,g_p(z)$).  Set $R(z):=R_0(z,\g(z))$. This is a linear form in $g_1(z),\ldots,g_w(z)$. 
	We infer from \eqref{eq:slem} that $R(z),R'(z),\ldots,R^{(w-1)}(z)$ have all valuation at least $wn -\varepsilon_0 n$. Since $\varepsilon_0<1$, 
	Shidlovslii's lemma (see, for instance, \cite[Lem.\ 5.2]{FN}) implies that they are linearly independent over $\mathbb K(z)$, as soon as $n$ is large enough with respect to $\varepsilon_0$ and the constant involved in Shidlovskii's lemma, say $n\geq n_0$. 
	On the other hand, it is easy to see that the power series  $R_i(z,\g(z))$, $0\leq i\leq w-1$, generate the same $\mathbb K(z)$-vector space as $R(z),R'(z),\ldots,R^{(w-1)}(z)$, which ends the proof.  
	\end{proof}

\begin{Fact}\label{lem:indep_S_L}
Let $n\geq n_0$.  Then there exist a positive real number $c_2$, which depends on $d$ and $\varepsilon _0$, but not on $n$, and a positive integer $c_3 \leq \varepsilon_0 n+ c_2$ such that the linear forms
	$$
	L_1(\Y),\ldots,L_q(\Y),R_0(1,\Y),\ldots,R_{w+c_3-1}(1,\Y)
	$$
	have rank $p$. 
\end{Fact}

\begin{proof}
	Consider the matrix $\Gamma(z):=(\gamma_{i,j}(z))_{1\leq i ,j\leq p}$ defined by
	$$
	\gamma_{i,j}(z):=\left\{\begin{array}{ll} \text{the coefficient of $Y_j$ in $R_{i-1}(z,\Y)$,} & \text{ if } 1\leq i\leq w, \\
	\text{the coefficient in $Y_j$ of $\ell_{i-w}(z,\Y)$,} & \text{ if } w <i\leq p\,,\end{array}\right.
	$$
	where  $\ell_1(z,\Y),\ldots,\ell_q(z,\Y)$ are the linear forms defined in Section~\ref{sec:formes_fonc}.  
It has entries in $\OK[z]$ and satisfies
	\begin{equation}\label{eq:Gamma}
	\Gamma(z)\begin{pmatrix}
	g_1(z)\\ \vdots \\ g_p(z)
	\end{pmatrix}
	=\begin{pmatrix}
	R_0(z,\g(z)) \\ \vdots \\ R_{w-1}(z,\g(z))\\0 \\ \vdots \\ 0
	\end{pmatrix}\, \; \text{ and } \; \Gamma(z)\begin{pmatrix}
	Y_1\\ \vdots \\ Y_p
	\end{pmatrix}
	=\begin{pmatrix}
	R_0(z,\Y) \\ \vdots \\ R_{w-1}(z,\Y)\\ \ell_1(z,\Y) \\ \vdots \\ \ell_q(z,\Y)
	\end{pmatrix}\,.
	\end{equation}

It follows from Fact \ref{fact1}, that $R_0(z,\Y),\ldots,R_{w-1}(z,\Y)$ 
are linearly independent over $\mathbb K(z)$, as are 
$\ell_1(z,\Y),\ldots,\ell_q(z,\Y)$. Since the latter vanish at $\Y=\g(z)$, while  
the former remain linearly independent at $\Y=\g(z)$, we deduce that 
$$\ell_1(z,\Y),\ldots,\ell_q(z,\Y),R_0(z,\Y),\ldots,R_{w-1}(z,\Y)$$ are linearly independent over $\mathbb K(z)$. 
According to \eqref{eq:Gamma}, we get that 
$\Gamma(z)$ is invertible.
	
	\medskip Set $\Delta(z):=\det \Gamma(z)\neq 0$.
	On the one hand, we have that $\deg \Delta(z)\leq wn + \chi_1$ for some constant $\chi_1$ that does not depend on $n$, while, 
	on the other hand, inverting \eqref{eq:Gamma} and using that 
	$\val(R_k(z,\g(z))) \geq wn -\varepsilon_0 n - k -1$, we deduce that there exists $i$, $1\leq i\leq p$, such that
	$$
	\val(\Delta(z))  + \val( g_i(z))\geq w(n-1) -\varepsilon_0 n +1\,.
	$$
	It follows that the order of $\Delta(z)$ at $z=1$ is at most $\varepsilon_0 n + \chi_2$, for some positive real number $\chi_2$ that does not depend on $n$.  
	Let $\chi_3$ be the smallest integer such that  $\partial^{\chi_3}(\Delta)(1)\neq 0$. Hence 
	\begin{equation}\label{eq:rho34}
	\chi_3\leq \varepsilon_0 n + \chi_2\,. 
	\end{equation}
	
	By inverting $\Gamma$ using \eqref{eq:Gamma}, we obtain that
	$$
	\Delta Y_i=\sum_{j=0}^{w-1} \Delta_{i,j+1}R_{j} + \sum_{j=1}^q \Delta_{i,j+w} \ell_j\,, \quad \forall i, 1\leq i\leq p\,,
	$$
	where we let $\Delta_{i,j}$ denote the cofactors of $\Gamma$. Applying $\chi_3$ times the operator $T\partial$, we obtain
	$$
	T^{\chi_3} \partial^{\chi_3}(\Delta)Y_i = \sum_{k=0}^{\chi_3-1} \partial^k(\Delta) \mathcal L_{k,i} + \sum_{j=0}^{w+\chi_3-1} p_{i,j,k}R_{j} + \sum_{j=1}^q\sum_{k=0}^{\chi_3} q_{i,j,k}\ell_j^{[k]} \,,
	$$
where $p_{i,j,k},q_{i,j,k}\in \mathbb K[z]$, $\mathcal L_{k,i}\in \mathbb K[z,\Y]$ and $\ell_j^{[k]}=(T\partial)^{k}(\ell_j)$. Specializing at $z=1$, and using 
that $\partial^k(\Delta)(1) =0$ when $k<\chi_3$, we deduce that 
	$$
	T(1)^{\chi_3}\partial^{\chi_3}(\Delta)(1) Y_i = \sum_{j=0}^{w+\chi_3-1} \lambda_{i,j,k}R_{j}(1,\Y) + \sum_{j=1}^q\sum_{k=0}^{\chi_3} \gamma_{i,j,k}\ell_j^{[k]}(1,\Y)\,,
	$$
	with $\lambda_{i,j,k}, \gamma_{i,j,k}\in \mathbb K$. Since $\partial^{\chi_3}(\Delta)(1)\neq 0$ and $T(1)\neq 0$, we obtain that 
	$Y_i$ belongs to the $\mathbb K$-vector space spanned by the $R_j(1,\Y)$, $0 \leq j \leq w+\chi_3-1$ and the $\ell_j^{[k]}(1,\Y)$, $1\leq j \leq q$, $0\leq k \leq \chi_3$. Since this holds for all $i$, $1\leq i\leq p$, these linear forms have rank at least $p$. 
	
	On the other hand, for every $k\geq 0$, we have
	$$
	\ell_j^{[k]}(z,g_1(z),\ldots,g_p(z))=(T\partial)^k(\ell_j(z,g_1,\ldots,g_p))=(T\partial)^k(0)=0 \,.
	$$
	Thus, the forms $\ell_j^{[k]}(z,\Y)$, $1\leq j \leq q$, $0\leq k \leq \chi_3$, belong to $\psi({\rm ev}_1(\mathcal I_z(\f))\cap \mathbb K[\X]_{\leq \delta d})$ and hence 
	are linear combinations of the $\ell_j(z,\Y)$, $1\leq j\leq q$. It follows that the rank of 
	the linear forms $$R_0(1,Y),\ldots,R_{w+\chi_3-1}(1,\Y),\ell_1(1,\Y),\ldots \ell_q(1,\Y)$$ is equal to $p=w+q$. According to \eqref{eq:Ll} and \eqref{eq:rho34}, and setting $c_2:=\chi_2$ and $c_3:=\chi_3$, 
	this ends the proof.  
\end{proof}

We are now ready to end the proof of the proposition. 

\begin{proof}[Proof of Proposition \ref{prop:formes_Padé}]
According to Fact~\ref{lem:indep_S_L}, we can find $w$ independent linear forms among $R_0(1,\Y),\ldots,R_{w+c_3-1}(1,\Y)$, say 
$M_1(\Y),\ldots,M_w(\Y)$, such that 
$$
L_1(\Y),\ldots,L_q(\Y),M_1(\Y),\ldots,M_w(\Y)
$$
form a basis of the $\mathbb K$-vector space $\mathbb K Y_1+\cdots+\mathbb K Y_p$. This proves Property (c) of the proposition. 
The fact that the linear forms $M_1(\Y),\ldots,M_w(\Y)$ also satisfy (a) and (b) is classical (see, for instance,  \cite[p.\,218-219]{FN}). 
Indeed, when constructing the first linear form $R_0(z,\Y)$, one can choose $\varepsilon_0$ sufficiently small 
with respect to $\varepsilon$ and then $n$ sufficiently large with respect to $\varepsilon_0$, say $n\geq n_1\geq n_0$. 
Then Property (a) follows from the third inequality in \eqref{eq:slem}, while Property (b) follows from the second one. 
\end{proof}

\subsection{Linear forms coming from $\J$}\label{sec:formes_J}

In this section, we introduce two additional families of linear forms, which are associated with the ideal $\mathcal J$ defined in \eqref{eq:defJ}.  

\medskip

\begin{itemize}
	\item[\rm (c)] We complete the family $L_1(\Y),\ldots,L_q(\Y)$ into a basis of $\psi(\I_1(\f)_{\leq \delta d})$ by adding linear forms $L_{q+1}(\Y),\ldots,L_{r}(\Y)$\footnote{It follows from Beukers' Lifting Theorem \cite[Thm.\ 1.3]{Be06} that $r=q$, so there is in fact no need to construct these linear forms. 
	However, we prefer not to use this strong statement and show instead how our main result can be proved using the classical approach followed by Siegel, Shidlovskii and Lang.}.  They \emph{vanish} when specializing each $Y_i$ at $g_i(1)$, $1\leq i\leq p$.
		
	\medskip
	
	\item[\rm (d)] We complete $L_1(\Y),\ldots,L_r(\Y)$ into a basis of $\psi(\J_{\leq \delta d})$ by adding 
	$u=s-r$ linear forms $N_1(\Y),\ldots,N_u(\Y)$, with the property that 
	\begin{equation}\label{eq:fmin}
	\vert N_i(\g(1))\vert \leq \eta_2^{\delta d} \vert P(\f(1))\vert\,, \quad 1\leq i\leq u\,,
	\end{equation} where $\eta_2$ is a positive real number that 
	only depends on $f_1,\ldots,f_m$. 
\end{itemize}

\subsubsection{Construction of $L_{q+1},\ldots,L_r$}\label{sec:formes_val}

Reasoning as in the proof of Lemma~\ref{lem:indep_Li}, we can find a basis of $\psi(\I_1(\f)_{\leq \delta d})$ with coefficients in $\OK$ and whose  height 
 does not depend on $\delta$ nor on $d$. Since $L_1(\Y),\ldots,L_q (\Y)\in \psi(\I_1(\f)_{\leq \delta d})$ are linearly independent, we can chose $(r-q)$ linear forms $L_{q+1}(\Y),\ldots,L_{r}(\Y)$ among this basis 
so that $L_1(\Y),\ldots,L_r(\Y)$ form a basis of $\psi(\I_1(\f)_{\leq \delta d})$. 

\subsubsection{Construction of $N_1,\ldots,N_u$}\label{sec:formes_mesure}
We first observe that the vector space $\psi(\J_{\leq \delta d})$ is spanned by $L_1(\Y),\ldots,L_r(\Y)$ and all the linear forms of the form $\psi(\X^\bmu P(\X))$, where $\X^\bmu$ is a monomial with degree at most $d\delta - d$. Since by definition $\dim\J_{\leq \delta d}=s$, we can find $u=s-r$ linear forms $N_1(\Y),\ldots,N_u(\Y)$ 
among the linear forms $\psi(\X^\bmu P(\X))$ so that
$$
L_1(\Y),\ldots, L_r(\Y),N_1(\Y),\ldots,N_u(\Y)
$$
form a basis of $\psi(\J_{\leq d\delta})$. Furthermore, for every $i$, $1\leq i\leq u$, we have $N_i(\g(1))=\f(1)^{\bmu_i}P(\f(1))$ for some vector $\bmu_i$ 
whose norm is at most $d\delta - d$. The existence of $\eta_2$ satisfying the inequality given in (d) follows directly. 

\begin{rem}\label{rem:ro4}
	The linear forms $L_1,\ldots,L_r,N_1,\ldots,N_u$ have coefficients in $\OK$. The height of  each of the forms $L_1,\ldots,L_r$ is bounded independently of $d$ and $\delta$, say by $\rho_3$. 
	The height of each of the forms $N_1,\ldots,N_u$ is equal to $H(P)$.
\end{rem}

\subsection{End of the proof of Case~(E) of Theorem \ref{thm:measure_system}}

We first choose an integer $\delta\geq \delta_2$. 
By Lemma \ref{lem:7-8}, we thus have $w > vh$. Let us choose $\varepsilon$ small enough so that $w\geq(1+2\varepsilon)vh$. 
Note that this choice does not depend on $d$. Now we let $n\geq n_1$ be an integer which will be defined more precisely 
 later in the proof.   

Applying our previous results, we get that 
$L_1,\ldots,L_r,N_1,\ldots,N_u$ are linearly independent over $\mathbb K$. Hence we can add to these linear forms $p-r-u=p-s=v$ linear forms among $M_1,\ldots,M_w$ 
to obtain a basis of  $\mathbb K Y_1+\cdots +\mathbb K Y_p$. 
Reordering  $M_1,\ldots,M_w$ if necessary, we can assume that these forms are  $M_1,\ldots,M_v$.

 Set $H_2 :=\rho_2 n^{(1+\varepsilon)n}$, where $\rho_2$ is defined by Proposition \ref{prop:formes_Padé}. The latter implies that 
  $H(M_i)\leq H_2$ for all $i$, $1\leq i\leq v$. 
By definition of $\I_1(\f)$, we also have $L_i(g_1(1),\ldots,g_p(1))=0$ for every $i$, $1\leq i \leq r$. 
Using a classical argument, as in \cite[Inequality (9),\,p.\,215]{FN}, we obtain the existence of a positive real number $\rho_4$, that depends neither on $n$ 
nor on $\varepsilon$, such that 
	\begin{equation}\label{eq:valeur_formes}
 	\sum_{i=1}^{v} \frac{\vert M_i(g_1(1),\ldots,g_p(1)) \vert}{H_2}+\sum_{i=1}^{u} \frac{\vert N_i(g_1(1),\ldots,g_p(1)) \vert }{H(P)} \geq \rho_4\frac{\vert \Delta \vert }{\rho_3^rH_2^{v}H(P)^{u}}\,,
	\end{equation}
	where $\Delta:=\det(L_1,\ldots,L_r,M_1,\ldots,M_{v},N_1,\ldots,N_u)$ and $\rho_3$ is defined in Remark~\ref{rem:ro4}. Since $\Delta\in \OK$, we deduce that 
	\begin{equation}\label{eq:min_Delta}
	\vert \Delta \vert \geq \house{\Delta}^{1-h} \geq  \rho_5(H_2^{v}H(P)^{u})^{1-h} \,,
	\end{equation}
	for some positive real number $\rho_5$. 
	By  Proposition \ref{prop:formes_Padé}, we have 
	$$
	 \frac{\vert M_i(g_1(1),\ldots,g_p(1)) \vert}{H_2} \leq \frac{\rho_1}{\rho_2} n^{-wn}, \quad \forall i,\,1\leq i \leq v\,.
	$$
We deduce from 
\eqref{eq:fmin},  \eqref{eq:valeur_formes} and \eqref{eq:min_Delta} the existence of two positive real numbers 
$\rho_6$ and $\rho_7$, which can depend on $d$ but not on $n$ nor on $\varepsilon$, such that
	\begin{equation}\label{eq:minoP}
\vert P(\f(1)) \vert \geq \rho_6H(P)^{1-uh} n^{-(1+\varepsilon)nvh}  -  \rho_7 H(P)n^{-wn} \,.
\end{equation}
At this point, we choose  $n$ to be the smallest integer such that $n^{\varepsilon nvh}$ is larger than $\frac{2 \rho_7}{\rho_6}H(P)^{uh}$. Furthermore, 
we assume $H(P)$ to be large enough, so that $n\geq n_1$.  Then we get that 
$$
\vert P(\f(1)) \vert \geq \frac{\rho_6}{2}H(P)^{1-uh} n^{-(1+\varepsilon)nvh} \,.
$$
By minimality of $n$, and assuming  $H(P)$ to be large enough with respect to $d$, we deduce the existence of a positive real number 
$\eta_3$ (which depends on $\varepsilon$ but not on $d$) such that 
$n^{(1+\varepsilon)nvh}\leq H(P)^{\eta_3u}$. Then we infer from \eqref{eq:minoP} the existence of two positive  
real numbers  $\eta_4$ and $\rho_8>$ such that
$$
\vert P(\f(1)) \vert \geq \rho_8H(P)^{-\eta_4u}\, ,
$$
where $\eta_4$ does not depend on $d$ but $\rho_8$ does. Since, there are only finitely many polynomials of bounded degrees and height, enlarging $\rho_8$ (for each $d$) if necessary,  we can assume that the latter inequality holds for all $P$ such that $P(\f(1))\not=0$. Since, by Lemma \ref{lem:7-8}, $u\leq \eta_1\delta^td^t$ and neither $\delta$ nor $\eta_1$ depend on $d$, Case (E) of Theorem~\ref{thm:measure_system} is proved.
\qed


\subsection{Effectivity of the bounds in Case~(E) of Theorem \ref{thm:measure_system}}\label{sec:E_func_measures}
At each step of the proof above, the constants can be effectively bounded. We do not provide further details, as such quantifications are well known.  
However, we emphasize that the constants $\eta_1$ and $\rho_3$ do not explicitly appear in the classical proofs where $f_1, \ldots, f_m$ are assumed to be algebraically independent. Moreover, in these proofs, the parameter $\delta_1$ depends  only on $m$ and the degree of the number field $\mathbb{K}$, with this dependence given explicitly.  
Furthermore, one can take 
$\delta_2=\delta_1$.
In the more general setting we consider here, these constants depend on the ideal $\mathcal{I}_z(\f)$. More precisely, given a basis of $\mathcal{I}_z(\f)$, one can compute a Gröbner basis using Buchberger's algorithm (cf.\ \cite[Thm.\ 5.53]{BWK93}) and derive explicit bounds for $\delta_1, \delta_2, \eta_1$, and $\rho_3$. The main challenge lies in computing a basis for $\mathcal{I}_z(\f)$ itself (see the discussion in \cite{AF24_EM}).

Another delicate point is bounding the constant $n_0$ appearing in Fact \ref{fact1}, which originates from Shidlovskii's Lemma. In principle, this can be achieved using the methods from \cite{BB85, BCY04}. Once bounds for $\delta_1, \delta_2, \eta_1, \rho_3$, and $n_0$ are established, bounding the remaining constants
$$
\eta_2,\eta_3,\eta_4,\rho_1,\rho_2,\rho_4,\ldots,\rho_9 \;\; \mbox{and} \;\; n_1$$
follows from the results in \cite{Sh79}. Finally, we highlight (without proof) that, in Case (E) of Theorem \ref{thm:measure_system}, the constants can be chosen as
$$
C_1=e^{-e^{C_3 h^{2t}d^{2t}\log(hd)}},\quad \text{ and } C_2=C_4 h^{t+1}\,,
$$
where $C_3$ and $C_4$ are effectively computable constants that do not depend on $d$, $h$, or $t$. Notably, these bounds are similar to those given in \eqref{eq:bounds_ci}.

\section{Proof of Case (M$_q$) of Theorem~\ref{thm:measure_system}}\label{sec:M_func}

In this section, our goal is to establish Case (M$_q$) of Theorem~\ref{thm:measure_system}. 
We follow the proof of Case (M$_q$) of Theorem~\ref{thm:LGN}, as given by Becker and Nishioka, which is based on Elimination Theory. For the reader's convenience, we adhere as closely as possible to the exposition in Nishioka's book \cite[Chap.\ 4]{Ni_Liv}. We assume that the reader is familiar with \cite[Chap.\ 4]{Ni_Liv} and the proof of Case (M$_q$) of Theorem~\ref{thm:LGN} presented therein. We will not provide the full details of all the computations. 
Instead, we focus on the new argument that we introduce.

\subsection{Basics of Elimination Theory} 

With any homogeneous unmixed ideal $I \subset \mathbb{Z}[X_0, \dots, X_m] =: \mathbb{Z}[\mathbf{X}]$ (see \cite[p.\,119]{Ni_Liv} for a definition) and any point $\boldsymbol{\omega} \in \mathbb{C}^{m+1}$, we associate the quantities $h(I)$, $N(I)$, $H(I)$, $\| \boldsymbol{\omega}, I \|$, and $|I(\boldsymbol{\omega})|$, as defined in \cite[Chap.\,4]{Ni_Liv}. 

The first three somehow measure the complexity of the homogeneous unmixed ideal $I$. 
The integer $h(I)$ denotes the height of $I$, already introduced in Section \ref{sec:Efunc_dimension}. Heuristically, the quantity $N(I)$ (sometimes denoted $\deg(I)$) is related to the degrees of a minimal set of generators of $I$, while $H(I)$ is related to the absolute values of the coefficients of the polynomials in such  a set.

The non-negative real number
\[
    \| \boldsymbol{\omega}, I \| = \min \{ \| \boldsymbol{\beta} - \boldsymbol{\omega} \| : \boldsymbol{\beta} \in \mathbb{C}^{m+1} \setminus \{0\} \text{ is a zero of } I \} \in \mathbb{R}_{\geq 0} \cup \{\infty\}
\]
denotes the distance from $\boldsymbol{\omega}$ to the zero set of $I$. The quantity $|I(\boldsymbol{\omega})|$, sometimes referred to as the \textit{absolute value of $I$ at 
$\boldsymbol{\omega}$}, is linked to the absolute values of the evaluations of a minimal set of generators of $I$ at $\boldsymbol{\omega}$ (see \cite[Chap.\ 3, \S4]{NP01} for further details).

 The proof of Case (M$_q$) of Theorem~\ref{thm:measure_system} relies on two key results from Elimination Theory. The first is an \textit{arithmetic Bézout theorem} \cite[Prop.\ 4.11]{NP01}. Given a homogeneous prime ideal $\mathfrak{p} \subset \mathbb{Q}[\mathbf{X}]$, a homogeneous polynomial $Q \notin \mathfrak{p}$, and a point $\boldsymbol{\omega} \in \mathbb{C}^{m+1}$, this theorem guarantees the existence of a homogeneous unmixed ideal $I$ whose zero set coincides with the intersection of the zero sets of $\mathfrak{p}$ and $Q$. Moreover, it satisfies a bound of the form 
\begin{equation}\label{eq:Elimination_Bezout} 
\vert I(\boldsymbol{\omega})\vert \leq \Theta\left(N(\mathfrak{p}), H(\mathfrak{p}), \vert \mathfrak{p}(\boldsymbol{\omega})\vert, \Vert \boldsymbol{\omega}, \mathfrak{p} \Vert, \deg(Q), H(Q), \vert Q(\boldsymbol{\omega})\vert \right), 
\end{equation}
for some function $\Theta : (\mathbb{R}_{\geq 0})^7 \to \mathbb{R}_{\geq 0}$, which is independent of $\mathfrak{p}$, $Q$, $I$, and $\boldsymbol{\omega}$. Furthermore, the parameters $N(I)$ and $H(I)$ are bounded in terms of $h(\mathfrak{p})$, $N(\mathfrak{p})$, $H(\mathfrak{p})$, $\deg(Q)$, and $H(Q)$.

The second result is a \textit{Liouville-type inequality} for homogeneous unmixed ideals \cite[Prop.,4.13]{NP01}. Given such an ideal $I$ and a point $\boldsymbol{\omega} \in \mathbb{C}^{m+1}$, this inequality provides a lower bound of the form:
\begin{equation}\label{eq:Elimination_Liouville} \vert I(\boldsymbol{\omega})\vert \geq \Phi\left(h(I), N(I), H(I), \Vert \boldsymbol{\omega}, I \Vert\right), \end{equation}
for some function  $\Phi : (\mathbb{R}_{\geq 0})^4 \to \mathbb{R}_{\geq 0}$, which is independent of both $I$ and $\boldsymbol{\omega}$.

\subsection{Strategy of proof}

Let $\omeg=(1,f_1(\alpha),\ldots,f_m(\alpha))$. 
The first step is to establish a lower bound of the form
\begin{equation}\label{eq:Elimination_Liouville2}
\vert I(\omeg) \vert \geq \Psi(N(I),H(I),h(I))\,,
\end{equation}
for any homogeneous unmixed ideal $I$ satisfying $h(I)\geq m-t+1$, where $\Psi : (\mathbb R_{\geq 0})^3 \to \mathbb R_{\geq 0}$ is independent of both $I$ and $\omeg$. 
Notably, in contrast to \eqref{eq:Elimination_Liouville}, this lower bound does not depend on $\Vert \omeg,I \Vert$. It corresponds to Propositions \ref{prop:majo_ideal} and \ref{prop:majo_ideal_bis} below. 

We proceed by contradiction, assuming that there exists a homogeneous unmixed ideal $I$ that does not satisfy \eqref{eq:Elimination_Liouville2}, so that 
\begin{equation}\label{eq:Elimination_Liouville2N}
\vert I(\omeg) \vert < \Psi(N(I),H(I),h(I))\,.
\end{equation}
 Among such ideals, we select one with maximal height. Considering its primary decomposition, we may reduce to the case where $I$ is a prime ideal. Comparing \eqref{eq:Elimination_Liouville} with \eqref{eq:Elimination_Liouville2N}, we conclude that $\Vert \omeg,I \Vert$  is small. 

Next, assume we can construct a homogeneous auxiliary polynomial $Q \in \mathbb{Z}[\mathbf{X}]$ taking a small non-zero value at $\boldsymbol{\omega}$, relative to its degree and height. Applying the arithmetic Bézout theorem to $I$ and $Q$, we obtain a new ideal, denoted $J$, whose height exceeds that of $I$. If $\Psi$ and $Q$ are chosen appropriately, we can combine the upper bounds for $\Vert \omeg,I \Vert$ and $|I(\boldsymbol{\omega})|$ with Inequality \eqref{eq:Elimination_Bezout} (applied with  $J$ in place of of $I$ and $I$ in place of $\mathfrak p$) to show that $J$ also fails to satisfy \eqref{eq:Elimination_Liouville2}. This contradicts the fact the $I$ has maximal height among the homogeneous unmixed ideals that does not satisfy \eqref{eq:Elimination_Liouville2}, thereby proving \eqref{eq:Elimination_Liouville2}.

To construct the auxiliary polynomial $Q$, we begin by considering a Padé approximant of the form
$$
R(z)=R_0(z,1,f_1(z),\ldots,f_m(z))
$$
where $R_0 \in \Q[z,\X]$  is a homogeneous polynomial in $\X$.  
By applying Siegel's Lemma, we obtain the existence an $R_0$ such that $R(z)$ has a high order of vanishing at $0$, with the degree and height of $R_0$ bounded in terms of this vanishing order.  
Next, using the Mahler system \eqref{eq:MSystem}, we can express $R(z^{q^k})$ as a polynomial in $1,f_1(z),\ldots,f_m(z)$, denoted by $R_k(z,\X)$. 
The polynomial $R_k(\alpha,\X)$ takes a small value at $\omeg$, but it has coefficients in $\Q$. Subsequently, applying the descent lemma \cite[Lem.\ 4.1.9]{Ni_Liv}, we obtain a homogeneous auxiliary polynomial $Q_k(\X) \in \Z[\X]$ from each $R_k(\alpha,\X)$. This process is described in Propositions~\ref{prop:polynôme_aux} and \ref{prop:polynôme_aux2} below.

The final step is as follows. Let $\prescript{\rm h}{}{P}\in \Z[\X]$ denote the homogenized polynomial obtained from the polynomial $P$ of Theorem~\ref{thm:measure_system} by adding the variable $X_0$. We aim to establish a lower bound for  $\vert \prescript{\rm h}{}{P}(\omeg) \vert$. Let $\mathfrak p$ denote the homogeneous ideal associated with $\omeg$, i.e., 
the ideal of $\Z[\X]$ generated by the homogeneous polynomials $Q \in \Z[\X]$ such that $Q(\omeg)=0$. 
Consider the ideal $I$ given by applying the arithmetic Bézout theorem  to $\mathfrak p$ and $\prescript{\rm h}{}{P}$. We have $h(I)=h(\p)+1\geq m-t+1$, so  \eqref{eq:Elimination_Liouville2} holds. By comparing this with \eqref{eq:Elimination_Bezout} and noting that $\vert \mathfrak p(\omeg)\vert=\Vert \omeg, \mathfrak p\Vert=0$, we obtain the inequality 
\begin{equation*}
\Theta\left(N(\mathfrak p),H(\mathfrak p),0,0,\deg(\prescript{\rm h}{}{P}),H(\prescript{\rm h}{}{P}),\vert \prescript{\rm h}{}{P}(\omeg)\vert \right) \geq \Psi(N(I),H(I),h(I))\,.
\end{equation*}
Since $N(I)$, $H(I)$ and $h(I)$ are bounded in terms of  $\mathfrak p$, $\deg(\prescript{\rm h}{}{P})$ and $H(\prescript{\rm h}{}{P})$, and since the ideal $\mathfrak p$ depends only on $f_1,\ldots,f_m$ and $\alpha$, this provides a lower bound for $\vert \prescript{\rm h}{}{P}(\omeg)\vert$, which completes  the proof.

The main novelty, compared to the proof of case (M$_q$) in Theorem \ref{thm:LGN}, lies in the use of the arithmetic Bézout theorem in the last step. Under the assumption of Theorem~\ref{thm:LGN}, we have $\mathfrak p = (0)$ and one may take $I=(\prescript{\rm h}{}{P})$. In contrast, when $\mathfrak p \neq (0)$, we have $t<m$. Thus, taking $I=(\prescript{\rm h}{}{P})$ does not lead to a contradiction since \eqref{eq:Elimination_Liouville2} holds only when $h(I) \geq m-t+1$. 
Moreover, we must adapt  the Padé approximant construction to the case where the functions are not assumed to be algebraically independent. Specifically, we use a refined version of Nishioka's Multiplicity Lemma \cite[Thm.\ 4.3]{Ni_Liv}.

To derive finer bounds for $C_1$ and $C_2$, we analyze two different regimes based on whether  $H(P)$ is large with respect to $\deg(P)$ or not. 
The first regime corresponds to Propositions~\ref{prop:polynôme_aux} and \ref{prop:majo_ideal}, while the second corresponds to Propositions~\ref{prop:polynôme_aux2} and \ref{prop:majo_ideal_bis}.

\subsection{Proof of Case (M$_q$) of Theorem~\ref{thm:measure_system}}\label{sec:casemq}
We continue with the notation introduced in this section. Recall that $\X:=(X_0,\ldots,X_m)$.
We now present a refined version of Nishioka's Multiplicity Lemma \cite[Thm.\ 4.3]{Ni_Liv}. We recall that, as in Section~\ref{sec:padé}, $\val$ denotes the usual valuation associated with $z$. This following result is proved in \cite[Thm.\ V.1.1]{Fe19}.

\begin{lem}
	[Multiplicity Lemma]\label{lem:Multiplicity}
	Let $M,N\geq 1$ be two integers.
	Let $R_0 \in \C[z,X_1,\ldots,X_m]$ with $\deg_z (R_0)\leq M$ and total degree at most $N$ in $X_1,\ldots,X_m$. If $R(z):=R_0(z,f_1,\ldots,f_m)\neq 0$ then
	$$
	\val (R) = \mathcal O(MN^{t}) \,,
	$$
	where the constant in $\mathcal O(\cdot)$ depends only on $f_1,\ldots,f_m$.
\end{lem}

\begin{rem}
	Alternatively, this result can be proven with a small modification of the proof of \cite[Thm.\ 4.3]{Ni_Liv}. 
	Let $\mathfrak p_z\subset \C[z][\X]$ denote the homogeneous ideal generated by the homogeneous polynomials that vanishes at $(1,f_1,\ldots,f_m)$. 
	Let $\prescript{\rm h}{}{R_0}\in \C[z][\X]$ be the polynomial, homogeneous in $\X$, obtained from $R_0$. The key modification is to replace the ideal
	 $I$ used in \cite{Ni_Liv} with the ideal derived from the arithmetic Bézout theorem over $\C(z)$ \cite[Thm.\ 4.1.7]{Ni_Liv}, applied to $\mathfrak p_z$ and $\prescript{\rm h}{}{R_0}$. 
\end{rem}

Throughout the proof, the constants $\gamma_1,\gamma_2,\gamma_3,\ldots$ will denote positive real numbers that depend only on $f_1,\ldots,f_m$ and $\alpha$ while, for any $\lambda$, the constants $\rho_1(\lambda),\rho_2(\lambda)$ will depend on $f_1,\ldots,f_m$, $\alpha$ and $\lambda$. 
Recall that we define $\omeg:=(1,f_1(\alpha),\ldots,f_m(\alpha))$. The purpose of Propositions~\ref{prop:polynôme_aux} and \ref{prop:polynôme_aux2} is to construct 
the auxiliary polynomials corresponding to the first and second regimes, respectively.

\begin{prop}
	\label{prop:polynôme_aux}
	There exist positive integers $\gamma_1,\ldots,\gamma_{6}$ such that, for any $N\geq \gamma_1$ and any $k$ with $q^k\geq \gamma_2N^{t+1}$ there exists a homogeneous polynomial $Q_k\in \Z[\X]$ satisfying the following inequalities:
	\begin{equation}
	\label{eq:deg_height_Qk}
	\deg (Q_k) \leq \gamma_{3}N,\quad \log H(Q_k)\leq \gamma_{4}Nq^k\,,
	\end{equation} 
	\begin{equation}\label{eq:encadr_Qk}
	-\gamma_{5}N^{t+1}q^k \leq \log\left(\vert Q_k(\omeg) \vert \vert \omeg \vert^{-\deg (Q_k)}\right) \leq -\gamma_{6}N^{t+1}q^k\,,
	\end{equation}
	and
	\begin{equation}
	\label{eq:majo_Qk}
	\vert Q_k(\omeg) \vert \vert \omeg \vert^{-\deg (Q_k)} \leq H(Q_k)^{-1}(\deg (Q_k)+1)^{-(2m+2)}\,.
	\end{equation}
\end{prop}
\begin{proof}

	This proposition corresponds to \cite[Prop.\,4.4.7]{Ni_Liv}. The main difference is that we do not assume that the functions $f_1,\ldots,f_m$ are algebraically independent. 
	Thus, the integer $m$, which in the setting of \cite{Ni_Liv} corresponds to the transcendence degree of these functions, must be replaced by 
$t$ in \eqref{eq:encadr_Qk}. To accommodate this new situation, we assume, without loss of generality, that $f_1,\ldots,f_t$ are algebraically independent over $\Q(z)$.
	
	We proceed as in \cite[p.\,140]{Ni_Liv} but instead of constructing a Padé approximant for $1,f_1,\ldots,f_m$, we construct a Padé approximant for $1,f_1,\ldots,f_t$, 
	denoted by 
	$$
	R(z)=R_0(z,1,f_1(z),\ldots,f_t(z)) \,,
	$$
	where $R_0(z,\X) \in \Q[z,\X]$ is homogeneous in $\X$, and we have 
	$$\deg_z (R_0),\,\deg_{\X} (R_0) \leq N,\; \log H(R_0) =\mathcal O(N^{t+1}),\text{ and }\val (R)\geq \gamma_7 N^{t+1}\,,$$ 
	for some positive real number $\gamma_7$.
	Meanwhile, Lemma \ref{lem:Multiplicity} applied with $M=N$ implies that $\val (R) = \mathcal O(N^{t+1})$. 
	The remainder of the proof follows straightforwardly from the proof of \cite[Prop.\ 4.4.3]{Ni_Liv} (see also \cite[Prop.\ V.3.1]{Fe19} with $M = N$ for detailed computations). 

By enlarging $\gamma_1$ if necessary,  \eqref{eq:majo_Qk}  follows easily from \eqref{eq:deg_height_Qk} and \eqref{eq:encadr_Qk}. 
Note that the purpose of \eqref{eq:majo_Qk} is to ensure that we can later apply the arithmetic Bézout theorem \cite[Lem.\ 4.1.3]{Ni_Liv}. 
\end{proof}

The following result is analogous to \cite[Prop.\ 4.4.8]{Ni_Liv}.

\begin{prop}
	\label{prop:polynôme_aux2}
There exist positive real numbers $\gamma_{8},\ldots,\gamma_{12}$, with $\gamma_{11}\geq 1$, such that, for any $\delta \geq t+2$ and any $s \geq \gamma_{8}$,  there exists a homogeneous polynomial $Q \in \Z[X_0,\ldots,X_m]$ satisfying the following inequalities:
	\begin{equation*}
	\deg (Q) \leq \gamma_{9}s,\quad \log H(Q)\leq \gamma_{10}s^\delta\,,
	\end{equation*} 
	\begin{equation*}
	-\gamma_{11}s^{t+\delta} \leq \log\left(\vert Q(\omeg) \vert \vert \omeg \vert^{-\deg (Q)}\right) \leq -\gamma_{12}s^{t+\delta}\,,
	\end{equation*}
	and
	\begin{equation*}
	\vert Q(\omeg) \vert \vert \omeg \vert^{-\deg (Q)} \leq H(Q)^{-1}(\deg (Q)+1)^{-(2m+2)}\,.
	\end{equation*}
\end{prop}

\begin{proof} 
For any $\delta\geq t+2$ and $s$, let $k$ be the least integer such that $q^{k}\geq 2\gamma_2s^{\delta-1}$ and set $N := \lceil s\rceil$. If $\gamma_8$ is large enough and $s \geq \gamma_8$, then the integers $N$ and $k$ satisfy the assumptions of Proposition \ref{prop:polynôme_aux}. The result then follows from 
Proposition~\ref{prop:polynôme_aux} by setting $Q=Q_k$ and noting that $Nq^k=\mathcal O(s^\delta)$. There is no restriction in assuming that $\gamma_{11}\geq 1$.
\end{proof}

The following two propositions are analogs of \cite[Prop.\ 4.4.9 and 4.4.10]{Ni_Liv}, respectively. They provide lower bounds of the form \eqref{eq:Elimination_Liouville2}, which we will need for the first and second regimes.

\begin{prop}\label{prop:majo_ideal}
	There exist a positive real number $\gamma_{13}$ and, for each $\lambda$, a positive real number $\rho_1(\lambda)$ such that, for any real numbers $\lambda,D,H$ 
	satisfying 
	$$
	\lambda \geq \gamma_{13},\quad D\geq \rho_1(\lambda),\quad \log H \geq D^{2t+3}\,,
	$$
	and any non-zero unmixed homoegeneous ideal $I \subset \Z[\X]$ that satisfies the following four conditions:
	\begin{itemize}
		\item[\rm (i)]  $I \cap \mathbb Z = (0)$,
		\item[\rm (ii)]  $r:=m+1 - h(I)$ is such that $1\leq r \leq t$,
		\item[\rm (iii)]  $N(I) \leq \lambda^{t-r}D^{t-r+1}$,
		\item[\rm (iv)]  $\log H(I) \leq \lambda^{t-r}D^{t-r}\log H$,
	\end{itemize}
	the following inequality holds: 
	\begin{equation}\label{eq:mino_Iomega}
	\log \vert I(\omeg) \vert \geq -\lambda^rD^{r-1}(D\log H(I)+N(I)\log H)\,.
	\end{equation}
\end{prop} 

\begin{rem}
Condition (ii) is equivalent to $h(I)\geq m-t+1$. It is worth noting that Nishioka's theorem~\cite[Thm.\ 4.2.1]{Ni_Liv}, which asserts that ${\rm tr.deg}_{\Q}(\omeg)=t$, can be deduced from this proposition. Indeed, if we had ${\rm tr.deg}_{\Q}(\omeg) <t$, then the homogeneous  ideal $\mathfrak p$ associated with $\omeg$ would have height $m-{\rm tr.deg}_{\Q}(\omeg) \geq m-t+1$. Thus, it would satisfy Conditions {\rm (i)} to {\rm (iv)} of Proposition~\ref{prop:majo_ideal} for sufficiently large parameters $\lambda$, $D$, and $H$.  However, it would not satisfy \eqref{eq:mino_Iomega}, since we would have $\vert \mathfrak p(\omeg) \vert =0$ (see \cite[Cor.\ 4.10]{NP01}).
\end{rem}

\begin{proof}
	The proof is similar to the one in \cite[Prop.\ 4.4.9]{Ni_Liv}, but with   \cite[Prop.\ 4.4.7]{Ni_Liv} replaced by Proposition~\ref{prop:polynôme_aux}, and the exponent $m$ replaced by $t$. We proceed by contradiction. Fix $\lambda,D$, and $H$ satisfying the assumptions of the proposition with $\gamma_{13}$ and $\rho_1(\lambda)$ chosen large enough, to be specified later. Suppose that there exists an ideal $I$ satisfying Conditions (i) to (iv) but not \eqref{eq:mino_Iomega}. We choose such an ideal $I$ with the largest height. 
	
	We proceed as in the proof of \cite[Prop.\ 4.4.5]{Ni_Liv}\footnote{In \cite{Ni_Liv}, Nishioka proves Proposition 4.4.5, which deals with the study of Mahler systems with polynomial coefficients. Proposition~\ref{prop:majo_ideal} actually corresponds to Proposition 4.4.9 in \cite{Ni_Liv}, which deals with general systems, and for which the proof is omitted. This explains the slight differences in our computations compared to those in the proof of Proposition 4.4.5 in \cite{Ni_Liv}.}. Using the primary decomposition of $I$, we obtain the existence of a prime ideal $\p_0 \supset I$ such that
	$$
h(\p_0)=h(I),\,\,	N(\p_0)\leq  \lambda^{t-r}D^{t-r+1},\,\,	\log H(\p_0)\leq \lambda^{t-r}D^{t-r}(\log H + m^2D)\,,
	$$
	and
	$$
	\log \vert \p_0(\omeg) \vert \leq -\frac{1}{3}\lambda^rD^{r-1}(D\log H(\p_0)+N(\p_0)\log H)=:-X\,,
	$$
	providing that $D$ is large enough with respect to $\lambda$. 
	
	Let $\mu$ be such that $\lambda = \mu^{2t+2}$ and set $N:=\lfloor \mu^{2t}D\rfloor$.  Assuming that $D$ is large enough with respect to $\lambda$ and using \cite[Lem.\ 4.1.4]{Ni_Liv}, we obtain 
	\begin{equation*}\label{eq:min_X_p}
	\min \{X,-\frac{1}{2}\log \Vert \omeg,\p_0\Vert\}\geq \log H \geq D^{2t+3}\geq \gamma_2\gamma_5qN^{2t+2}\,.
	\end{equation*} 
Let $k$ be the largest integer such that $\gamma_5N^{t+1}q^k \leq \min \{X,-\frac{1}{2}\log \Vert \omeg,\p_0\Vert\}$.  
Then, if $D$ is large enough, we may consider the auxiliary polynomial $Q_k$ given by Proposition~\ref{prop:polynôme_aux}.  We apply the arithmetic Bézout theorem \cite[Lem.\ 4.1.3]{Ni_Liv} with $\p_0$ and $Q_k$. If $h(I)<m$, this provides an ideal $J$ with $h(J)>h(I)$. Arguing as in the final part of the proof of \cite[Prop.\ 4.4.5]{Ni_Liv}, if $\lambda$ were chosen large enough, one can check that $J$ satisfies Conditions (i) to (iv) but not \eqref{eq:mino_Iomega}. This contradicts the maximality of $h(I)$. Therefore, $h(I)=m$. This implies that $r=1$ and  \cite[Lem.\ 4.1.3]{Ni_Liv}  leads to a contradiction, completing the proof.  
\end{proof}

\begin{prop}\label{prop:majo_ideal_bis}
	There exist a positive real number $\gamma_{14}$ and, for each $\lambda$, a positive real number $\rho_2(\lambda)$ such that, for any real numbers $\lambda,\delta,D,H$ satisfying 
	$$
	\lambda \geq \gamma_{14},\quad t+2\leq \delta \leq 2t+3,\quad D\geq \rho_2(\lambda),\quad \log H = D^{\delta}\,,
	$$
 and any non-zero unmixed homogeneous ideal $I \subset \Z[\X]$ satisfying the following four conditions:
	\begin{itemize}
		\item[\rm (i)]   $I \cap \mathbb Z = (0)$,
		\item[\rm (ii)]   $r:=m+1 - h(I)$ is such that $1\leq r \leq t$,
		\item[\rm (iii)]   $N(I) \leq \lambda^{t-r}D^{t-r+1}$,
		\item[\rm (iv)]   $\log H(I) \leq \lambda^{t-r+\delta-1}D^{t-r}\log H$,
	\end{itemize}
	the following inequality holds: 
	\begin{equation}\label{eq:mino_Iomega_bis}
	\log \vert I(\omeg) \vert \geq -\lambda^{r-1}D^{r-1}(\lambda D\log H(I)+\lambda^{\delta}N(I)\log H)\,.
	\end{equation}
\end{prop}

\begin{proof}
	The proof is analogous to the one in \cite[Prop.\ 4.4.10]{Ni_Liv}, with the exponent $m$ replaced by $t$, and Proposition \ref{prop:polynôme_aux2} replacing 
	\cite[Prop.\ 4.4.8]{Ni_Liv}.
\end{proof}

We are now able to prove  Case (M$_q$) of Theorem~\ref{thm:measure_system}.

\begin{proof}[Proof of Case (M$_q$) of Theorem~\ref{thm:measure_system}]
Consider the homogeneous polynomial $\prescript{\rm h}{}{P}\in \Z[\X]$ obtained from $P$. Thus, we have $\prescript{\rm h}{}{P}(\omeg)=P(f_1(\alpha),\ldots,f_m(\alpha))$, 
$\deg(\prescript{\rm h}{}{P})=\deg(P)$, and $H(\prescript{\rm h}{}{P})=H(P)$. Assume that $\vert \prescript{\rm h}{}{P}(\omeg) \vert\neq 0$. We aim to bound $\vert \prescript{\rm h}{}{P}(\omeg) \vert$ from below. Let $\p$ denote the homogeneous ideal associated with $\omeg$, i.e., the homogeneous ideal generated by
	$$
	\{Q\in \Z[\X] \text{ homogeneous}\,:\, Q(\omeg)=0\}.
	$$
	This is a prime ideal, and $h(\p)\geq m-t$\footnote{In fact, by Nishioka's theorem \cite[Thm.\ 4.2.1]{Ni_Liv}, $h(\p)=m-t$, but  we do not need to use this fact explicitly in the proof.}.  	
	 Additionally, $N(\p)$ and $H(\p)$ are positive real numbers that depend  only on $\omeg$. Since $\omeg$ is a zero of $\p$, we also have $\Vert \omeg,\p \Vert = 0$ and $\vert \p(\omeg)\vert = 0$ (see \cite[Cor.\,4.10]{NP01}). We may assume that
	$$
	\vert \prescript{\rm h}{}{P}(\omeg) \vert \vert \omeg \vert^{-d} \leq H(P)^{-1}(d+1)^{-2m-2},
	$$
	where $d:=\deg(P)$, otherwise, there is nothing to prove. 
	
	We apply \cite[Lem.\,4.1.3]{Ni_Liv} to $\p$ and $\prescript{\rm h}{}{P}$. In this case, we take 
	$$
	X = -\log \left(\vert \prescript{\rm h}{}{P}(\omeg)\vert \vert \omeg \vert^{-d}\right) >0
	$$
	and $\sigma =1$. Since $\prescript{\rm h}{}{P}(\omeg)\neq 0$ we have $X < +\infty$. It follows from the arithmetic Bézout theorem \cite[Lem.\,4.1.3]{Ni_Liv} that there exists an unmixed homogeneous ideal $I$ such that $h(I)=h(\p)+1$, $I \cap \Z = (0)$, and
	\begin{eqnarray}
	\nonumber N(I)&\leq& \gamma_{15} d
	\\  \label{eq:mino_intersec_ideals} \log H(I) &\leq& \gamma_{15}  d +  \gamma_{16} \log H(P)
	\\  \nonumber  \log \vert I(\omeg) \vert & \leq& \frac{\log \vert \prescript{\rm h}{}{P}(\omeg)\vert}{2}+\gamma_{17}d + \gamma_{16} \log H(P)\,.
	\end{eqnarray}
	Set $r:=m+1-h(I)$, so that
	$$
	r = m+1 - h(I) = m+1 - (h(\p)+1)=m-h(\p)\leq t\,.
	$$
	Now, let $\lambda:=\max\{\gamma_{13},\gamma_{14}\}$. If necessary, enlarge $\gamma_{15}$ so that 
	$$\gamma_{15} \geq \max\{\rho_1(\lambda),\rho_2(\lambda)\}\,.$$ We now consider two different regimes, according to the size of $\log H(P)$ relative to $d$. 
	
	\subsubsection*{First regime.} We assume that $\gamma_{15}  d +  \gamma_{16} \log H(P) \geq (\gamma_{15} d)^{2t+3}$. In that case, we set $D:= \gamma_{15} d$ and let $H$ be such that $\log H:=\gamma_{15}  d +  \gamma_{16} \log H(P)$.
It can be verified that $I$ satisfies Conditions (i) to (iv) of Proposition \ref{prop:majo_ideal}. Thus, by \eqref{eq:mino_Iomega} and \eqref{eq:mino_intersec_ideals} ,  we obtain 
	\begin{equation}\label{eq:minoP_1}
	\log \vert I(\omeg) \vert \geq   -\lambda^r D^{t-1}(D\log H(I)+N(I)\log H) \geq -\gamma_{18} (d^t\log H(P) + d^{t+1})\,.
\end{equation}
This gives the first bound.

\subsubsection*{Second regime} We assume that  $\gamma_{15}  d +  \gamma_{16} \log H(P) < (\gamma_{15} d)^{2t+3}$. In this case, set
	$$
	\delta := \max\left\{t+2\; ; \; \frac{ \log(\gamma_{15}  d +  \gamma_{16} \log H(P))}{\log(\gamma_{15}d)}\right\} \leq 2t+3\,,
	$$
	 $D:=\gamma_{15} d$ , and let $H$  be defined by $\log H = D^\delta$. One can check  that $I$ satisfies Conditions (i) to (iv) of Proposition~\ref{prop:majo_ideal_bis}. 
	 Applying  \eqref{eq:mino_Iomega_bis} and \eqref{eq:mino_intersec_ideals}, we obtain 
\begin{align*}
	\log \vert I(\omeg) \vert &\geq -\lambda^{r-1}D^{r-1}(\lambda D\log H(I)+\lambda^{\delta}N(I)\log H)
	\\ & \geq -\gamma_{19}(d^t\log H(P) + d^{t+\delta})\,.
\end{align*}
We distinguish two cases according to the value of $\delta$. If $\delta=t+2$, we obtain 
$$
	\log \vert I(\omeg) \vert \geq -\gamma_{20}(d^t\log H(P) +d^{2t+2})\,.
$$
If $\delta=\frac{ \log(\gamma_{15}  d +  \gamma_{16} \log H(P))}{\log(\gamma_{15}d)}$, 
we have that $d^\delta= \mathcal O(d+\log H(P))$ and we obtain 
$$
	\log \vert I(\omeg) \vert \geq -\gamma_{21}(d^t\log H(P) +d^{t+1})\,.
$$
We thus deduce that 
\begin{equation}\label{eq:minoP_2}
	\log \vert I(\omeg) \vert \geq -\gamma_{22}(d^t\log H(P) +d^{2t+2})\,.
\end{equation}
This provides the second bound. 
	
\subsubsection*{Conclusion} From \eqref{eq:mino_intersec_ideals}, \eqref{eq:minoP_1}, and \eqref{eq:minoP_2}, it follows that 
\begin{equation}\label{eq:minPMQ}
\vert \prescript{\rm h}{}{P}(\omeg)\vert =\big| P(f_1(\alpha), \dots, f_m(\alpha)) \big| \geq e^{-\gamma d^{2t+2}} H(P)^{-\gamma d^t}\, ,
\end{equation}
for some positive real number $\gamma$. 
Case (M$_q$) of Theorem~\ref{thm:measure_system} then follows by setting
$C_2:= \gamma$ and $C_1 := e^{-C_2 d^{2t+2}}$. 
\end{proof}

\section{Proof of Theorem~\ref{thm:measure_sans_system}}\label{sec:remove_sing}

To derive Theorem~\ref{thm:measure_sans_system} from Theorem~\ref{thm:measure_system}, we must identify a suitable linear system that involves the functions 
$f_1,\ldots,f_r$, along with possibly some additional functions, such that the point 
$\alpha$ is a regular point for this system.

The following result provides the necessary  background for applying Theorem~\ref{thm:measure_system}.

\begin{prop}\label{prop:good_eq} The two following results hold. 
{$\,$}
	\begin{itemize}
		\item[\rm (E)] Let $f$ be an $E$-function. Then $f$ is a solution to a linear differential equation over $\Q(z)$ whose  singularities are confined to $\{0,\infty\}$. 
		
		\item[\rm (M$_q$)] Let $f$ be an $M_q$-function and let $\alpha \in \Q^\times$ with $\vert \alpha \vert < 1$. Assume that $f$ is well-defined at $\alpha$ and 
		that $\ell$ is a positive integer satisfying  
		$\vert \alpha^{q^\ell} \vert < \rho$,  
		where $\rho>0$ denotes the radius of convergence of $f$.  Then $f$ is a solution to a $q^\ell$-Mahler equation for which $\alpha$ is a regular point. 
	\end{itemize}
\end{prop}

\begin{proof}
	Case (E) corresponds to \cite[Thm.\ 4.3]{An00}. Case (M$_q$) follows from the proof of \cite[Prop.\ 2.5]{AF24_EM}. This proposition states that an $M_q$-function 
	which is analytic on some centered disk with radius $R<1$ is a solution to a $q$-Mahler equation with no singularities on this disk. In fact, a minor modification of the proof implies the following: \textit{if $f$ is well-defined at $\alpha^{q^k}$ for every integer $k\geq 0$, then $f$ is  a solution to a $q$-Mahler equation for which $\alpha$ is regular}. 
	Since $f$ is also an $M_{q^\ell}$-function and since the assumption made on $\ell$ ensures that  $f$ is well-defined at $\alpha^{(q^\ell)^k}$ for every $k\geq 0$, this result 
	implies the existence of a $q^\ell$-Mahler equation satisfied by $f$ and for which $\alpha$ is regular.
	\end{proof}

\begin{rem} 
The proof of Case (E) does not rely on Beukers' Lifting Theorem \cite[Thm.\ 1.3]{Be06}, but instead on earlier work by André \cite{An00}, which transfers to $E$-functions some 
properties of $G$-functions established by the Chudnovsky brothers \cite{CC83}. 
In contrast, the proof of Case (M$_q$) depends on the Lifting Theorem for  $M_q$-functions \cite[Thm.\ 1.4]{AF17}, which serves as an analogue of Beukers' Lifting Theorem 
 for $M_q$-functions. 
\end{rem}

We are now ready to prove Theorem~\ref{thm:measure_sans_system}. 

\begin{proof}
	[Proof of Theorem \ref{thm:measure_sans_system}]
	We start with Case (E). Let $\partial$ denote the differential operator. For each $f_i$, $1\leq i\leq r$, let us 
	consider a differential equation given by Proposition~\ref{prop:good_eq}, say 
	\begin{equation}\label{eq:goodE}
	a_{i,0(z)}f_i(z)+a_{i,1}(z)\partial f_i(z)+a_{i,2}\partial^2f_i(z) + \cdot + a_{i,m_i}(z)\partial^{m_i}f_i(z)=0\,.
	\end{equation}
	Hence, for every $i$, $1\leq i\leq r$, $a_{i,m_i}(z)=z^{d_i}$ for some  integer $d_i\geq 0$. 
	The column vector $\prescript{\rm t}{}(f_i(z),\partial f_i(z),\ldots,\partial^{m_i-1}f_i(z))$ is solution to the linear system $\partial Y(z)=A_i(z)Y(z)$ where
	\begin{equation}\label{eq:Ai}
	A_i(z)=\begin{pmatrix}
	0 & 1 & 0 & \cdots & 0
	\\ \vdots & \ddots & \ddots & &\vdots
	\\ \vdots & & \ddots & \ddots & 0
		\\ 0 & \cdots & \cdots & 0 & 1 
	\\ -\frac{a_{i,0}(z)}{a_{i,m_i}(z)}& \cdots &\cdots& \cdots & -\frac{a_{i,m_{i}-1}(z)}{a_{i,m_i}(z)}
	\end{pmatrix}\,
	\end{equation}
	and the point $\alpha$ is regular with respect to this system. Then the functions
	$\partial^k f_{i}$, $1\leq i \leq r$, $0 \leq k \leq m_i-1$,  form the coordinates of a vector solution to  the differential system $\partial Y(z)=A(z)Y(z)$, where
	$$
	A(z)=A_1(z)\oplus \cdots \oplus A_r(z)=\begin{pmatrix}
	A_1(z) \\ & \ddots \\ &&A_r(z)
	\end{pmatrix}
	$$
	and the point $\alpha$ remains regular with respect to this system. Furthermore, the transcendence degree of the field extension of $\Q(z)$ generated by all these functions is equal to $\tau$. The  result thus follows directly from Theorem~\ref{thm:measure_system}.
	
	\medskip

We prove Case (M$_q$) in the same way. Let $\ell$ be an integer large enough, so that $\vert \alpha^{q^\ell}\vert$ is smaller than each of the radius of convergence of the $f_i$, $1\leq i\leq r$. 
By Proposition~\ref{prop:good_eq},  each $f_i$ is solution to a $q^\ell$-Mahler equation, say 
	\begin{equation}\label{eq:goodM}
	a_{i,0}(z)f_i(z)+a_{i,1}(z)f_i(z^{q^\ell})+ \cdot + a_{i,m_i}(z)f_i(z^{q^{\ell m_i}})=0\,,
	\end{equation}
	which is regular at $\alpha$. 
	Then the column vector $\prescript{\rm t}{}(f_i(z),f_i(z^{q^\ell}),\ldots,f_i(z^{q^{\ell m_{i-1}}}))$ is solution to the $q^\ell$-Mahler system $Y(z^{q^\ell})=A_iY(z)$, 
	where $A_i$ is defined as in \eqref{eq:Ai}. Furthermore, the point $\alpha$ is regular with respect to this system. Then the functions
	$f_{i}(z^{q^{\ell k}})$, $1\leq i \leq r$, $0 \leq k \leq m_i-1$, form the coordinates of a vector solution to  the $q^\ell$-Mahler system $Y(z^{q^\ell})=AY(z)$, 
	where $A(z)=A_1(z)\oplus \cdots \oplus A_r(z)$, and again the point $\alpha$ is regular with respect to this system. Furthermore, the transcendence degree of the field extension of $\Q(z)$ generated by all these functions is at most $\tau$. The  result thus follows directly from Theorem~\ref{thm:measure_system}.
\end{proof}

\section{State of the art}\label{sec:history}

In this final section, we present a brief history of the previous results related to the present work.
\subsection{The case of $E$-functions}

We begin with results concerning the values of $E$-functions. 

\subsubsection{Algebraic independence measures}

Algebraic independence measures for the values of particular $E$-functions at algebraic points have long been known. Siegel \cite{Si29} first obtained such a measure in the case of the Bessel function $J_0(z)$ and its derivative. A few years later, Mahler \cite{Ma32} derived a similar measure for the values of the exponential function evaluated at several linearly independent algebraic points. The bounds obtained in these two results have the same form as the one given in Theorem~\ref{thm:LGN} and can therefore be considered 
  as precursors to it. When Shidlovskii proved his famous theorem, it was clear to expert in the field that his proof had to lead to a measure of algebraic independence. Such a result was first obtained by Lang \cite{La62}, who proved Case~(E) of Theorem~\ref{thm:LGN}.

After Lang proved his theorem, subsequent work focused on refining the constants $C_1$ and $C_2$ that appear in it. Such refinements were achieved by Galochkin \cite{Ga68}, Nesterenko \cite{Ne77}, Shidlovskii \cite{Sh79,Sh82}, Brownawell \cite{Br85}, and Gorelov \cite{Gor90}, to name just a few. Lang had already obtained an explicit value for $C_2$, and these results progressively lowered $C_2$ and made $C_1$ more effective (see \cite[Chap.\ 5, \S 5.2]{FN}). 
 In Case (E) of Theorem~\ref{thm:LGN}, according to \cite[Thm.\ 1]{Gor90}, we can now take 
\begin{equation}\label{eq:bounds_ci}
C_1 = e^{-e^{C_3\deg(P)^{2m}\ln(\deg(P)+1)}}\quad \text{ and } \quad C_2=\sqrt{m}4^mh^{m+1}\,,
\end{equation}
where $h$ is the degree of number field spanned by the coefficients of the $E$-functions and 
the point $\alpha$, and $C_3$ is a positive real number depending only on $f_1,\ldots,f_m$ and $\alpha$. 
Obtaining an effective value for $C_3$ has long been a challenge, as it depends on the Shidlovskii constant, a parameter appearing in the Shidlovskii lemma 
(cf.\ \cite[Lem.\ 5.2, p.221]{FN}). This goal was finally achieved thanks to the work of Bertrand and Beukers \cite{BB85} and Bertrand, Chirskii, and Yebbou \cite{BCY04}.

In another direction, Shidlovskii \cite[Chap.\ 12, \S 4]{Sh_Liv} provided a series of results generalizing Theorem~\ref{thm:LGN} for algebraically independent $E$-functions that do not satisfy Assumption \ref{Assump_ii} (cf.\ p.\ 2). However, he did not express his results as an alternative, as we do in Theorem~\ref{thm:measure_system}. As a result, he had to impose certain technical conditions to ensure the algebraic independence of the values he considered. These technical conditions prevent him from deducing a clean statement as in Theorem~\ref{thm:measure_system}.

Recently, Fischler and Rivoal \cite{FR23} were able to derive a \textit{linear independence measure} for the values of general $E$-functions from a similar measure obtained by Shidlovskii for the values of $E$-functions with rational coefficients at non-zero rational points  \cite[Thm.\ 1, p.\ 358]{Sh_Liv} and Beukers' Lifting Theorem \cite{Be06}. Like our result, theirs does not require the (linear) independence of the functions and is instead expressed as an alternative. 
Since $E$-functions form a ring and a polynomial can always be viewed as a linear form in monomials, such a general linear independence measure leads to an algebraic independence measure similar to the one given in Theorem~\ref{thm:measure_system}. However, the algebraic independence measure obtained this way is weaker than the one we obtain (see also Section~\ref{sec:E_func_measures}). In particular, the effectivity of $C_1$ in that case remains unclear. Moreover, their proof relies on Beukers' theorem \cite{Be06}, while our proof of Theorem~\ref{thm:measure_system} follows the classical approach of Siegel, Shidlovskii and Lang.

\subsubsection{Irrationality and transcendence measures}

The first results concerning the part of Conjecture~\ref{conj} related to the set $\mathbf E$ were established by Siegel \cite{Si29} in 1929. The only known case of the conjecture pertains to the values of $E$-functions that are solutions to (possibly inhomogeneous) first-order equations. This was shown by Lang \cite{La62} in the homogeneous case and can be found in Galochkin \cite{Ga68} for the inhomogeneous case, which does not requires any new idea. All other results either demonstrate that the number in question is not a Liouville number, or that it is either an $S$-number or a $T$-number. Numerous partial results, derived from those mentioned in the previous section, follow in this vein. 
The first general result for all elements of $\mathbf{E}$ is recent and due to Fischler and Rivoal \cite{FR23}. They proved that no Liouville number belongs to $\mathbf{E}$ and, more generally,  that all transcendental elements of $\mathbf{E}$ are either $S$- or $T$-numbers.

Our approach provides an alternative proof of their results, as follows. Let $f(z)$ be an arbitrary $E$-function and $\alpha$ be an arbitrary non-zero algebraic number. Furthermore, assume that $f(\alpha)$ is transcendental. For every positive integer $d$, we apply Case (E) of Theorem~\ref{thm:measure_sans_system} with $r=1$ to deduce that there exist two positive real numbers $A$ and $B$ such that the inequality
$$
\vert P(f(\alpha))\vert >  A H(P)^{B} \,,
$$
holds for all non-zero polynomials $P \in \mathbb{Z}[X]$ of degree $d$. Thus, we have $w_d(f(\alpha))<\infty$. Since algebraic numbers are not Liouville numbers, taking $d=1$, we obtain that no element in $\mathbf E$ is a Liouville number, proving one part of Corollary~\ref{coro:Lnumbers}. More generally, we conclude that all transcendental numbers in $\mathbf{E}$ are either $S$-numbers or $T$-numbers. Furthermore, in the case where $\tau = 1$ in Theorem~\ref{thm:measure_sans_system}, we can infer that $f(\alpha)$ is an $S$-number. This last result slightly generalizes that of Lang and Galochkin mentioned above.

\subsection{The case of $M$-functions}

We now  recall some results concerning the values of $M$-functions.

\subsubsection{Measures of algebraic independence}

Mahler's method was introduced by Mahler in a series of three papers \cite{Ma29,Ma30a,Ma30b}. At that time, $M$-functions had not yet been defined, but some of Mahler's results retrospectively prove the algebraic independence of the values of certain $M$-functions at non-zero algebraic points. The first transcendence measures were obtained in the early 1980s by Galochkin \cite{Gal80}, Millet \cite{Mil82}, and Becker \cite{Bec86}, who quantified Mahler's approach. Independently, Becker \cite{Bec_Announc, Bec88} and Nesterenko \cite{Ne85} provided the first algebraic independence measures for the values of several $M_q$-functions, each of which is a solution to an inhomogeneous equation of order 1. 
In contrast to Becker's work, Nesterenko \cite{Ne85} did not follow Mahler's original method. Instead, he developed an entirely new approach based on Elimination Theory. This innovation allowed him to achieve significantly better bounds than Becker \cite{Bec88}. Moreover, as noted by Wass \cite{Wa_these}, the restriction to first-order equations is unnecessary. Ku. Nishioka \cite{Ni90} subsequently extended Nesterenko's method to general linear Mahler systems. Her result, known as \emph{Nishioka's Theorem} \cite[Thm.\ 4.2.1]{Ni_Liv}, serves as the analogue of the Siegel-Shidlovskii Theorem within the framework of $M$-functions.\footnote{In light of the discussion above, it might be more accurate to refer to this result as the \emph{Nesterenko--Nishioka Theorem}.} 
In the same year, Nishioka \cite{Ni90_2} established a multiplicity lemma that is, in some sense, analogous to the Shidlovskii Lemma for $E$-functions. This result enabled Becker \cite{Bec91} to provide the first proof of Case (M$_q$) of Theorem~\ref{thm:LGN}. Subsequently, Nishioka \cite{Ni91} refined Becker's bounds for $C_1$ and $C_2$ using the same approach. These bounds have not been significantly improved since: $C_2$ can, in principle, be effectively bounded (though it seems that no one has yet provided such a bound), and we can take 
\begin{equation}\label{eq:bounds_ci_M} C_1 = e^{-C_2d^{2m+2}}\,.
\end{equation}
The bound obtained in \eqref{eq:minPMQ} has the same form, with 
$m$ replaced by $t$. 

Beyond Case (M$_q$) of Theorem~\ref{thm:LGN}, several algebraic independence measures have been established in broader contexts, for example, for $M$-functions that do not necessarily satisfy Assumption \ref{Assump_i} given on page \pageref{Assump_i}. These results could also be used to deduce Case (M$_q$) of Theorem~\ref{thm:measure_system}. In this direction, Philippon \cite{Ph98} introduced the concept of a $K$-function and showed that $M$-functions are particular examples of $K$-functions. In \cite[Thm.\,6]{Ph98}, he provided an algebraic independence measure for linear Mahler systems with polynomial coefficients. 
In 2013, Zorin \cite{Zo13} released a preprint on certain generalizations of $M$-functions, providing algebraic independence measures. In the original version, the main result concerning $M$-functions was incorrect. After a discussion with the first author, Zorin released several revised versions of his preprint (in 2016 and 2017), in which he announced part of Corollary~\ref{coro:Lnumbers} related to the set $\mathbf M$ and, more generally, that $\mathbf M$ contains no $U$-number. His general result also implies an algebraic independence measure similar to Case (M$_q$) of Theorem~\ref{thm:measure_system}, although it is weaker than the one obtained in this paper. Furthermore, this revised version was never published in an academic journal, leaving the status of its proof uncertain. 
Finally, in her PhD thesis, Fernandes \cite{Fe19} proposed a measure in the spirit of Theorem~\ref{thm:measure_system}, but it is weaker than the one in \eqref{eq:minPMQ}  and relies on a result by Jadot \cite{Ja96}, which has also not been published in an academic journal. 
Note that all the results just mentioned ultimately rely on Elimination Theory, although their presentations differ from the one given in Nishioka's monograph \cite{Ni_Liv}. In proving Theorem~\ref{thm:measure_system}, we chose to follow the approach presented in \cite{Ni_Liv}, which provides more details than \cite{Ja96, Ph98, Zo13, Fe19}, in order to make our argument easier for the reader to follow. We also want to emphasize that the proof of Theorem~\ref{thm:measure_system} does not require any new tools beyond those introduced by Nesterenko and used by Becker and Nishioka when they obtained the first versions of Case~(M$_q$) of Theorem~\ref{thm:LGN}, and that Theorem~\ref{thm:measure_system} could thus have been obtained much earlier using the existing techniques.


\subsubsection{Irrationality and transcendence measure}

The problem addressed in Corollary~\ref{coro:Lnumbers} has a long history, which we will briefly outline here. In his first paper on the subject in 1929, Mahler \cite{Ma29} noted, although without proving it, that some of the numbers he had just shown to be transcendental could not be Liouville numbers. Conjecture~\ref{conj} for elements of $\bf M$ is considered folklore and is difficult to trace, but it is clear that it was known to all specialists in the field. The only known case of the conjecture concerns the values of $M$-functions that are solutions to (possibly inhomogeneous) first-order equations, which was proved by Galochkin \cite{Gal80} in 1980. 
All other results either showed that the number under study is not a Liouville number or that it is either an $S$-number or a $T$-number. 

Some specific cases have been studied explicitly. In 1993, Shallit conjectured that no automatic real number could be a Liouville number (the conjecture later appeared in print in \cite{Sh99}). In addition, in his correspondence with Shallit, Becker conjectured that transcendental automatic real numbers are $S$-numbers. These conjectures concern a particular class of $M$-functions, namely the generating series of automatic sequences, evaluated at specific rational points of the form $1/b$, where $b \geq 2$ is an integer. Shallit's conjecture was proved in 2006 by the first author and Cassaigne \cite{AC06}, without using Mahler's method but rather employing the approach based on the Schmidt subspace theorem introduced in \cite{ABL,AB07}. Using the same method, the first author and Bugeaud \cite{AB11} obtained a transcendence measure for irrational automatic real numbers in 2011, which implies that these numbers are either $S$-numbers or $T$-numbers in Mahler's classification.

In 2013, Zorin \cite{Zo13} released a preprint (see earlier discussion), claiming that no element of $\bf M$ can be a $U$-number. In 2015, Bell, Bugeaud, and Coons \cite{BBC} extended the result of \cite{AC06} to include the values of any $M$-function with rational coefficients, evaluated at a rational point of the form $1/b$. These authors also proved that values of generating functions of regular sequences at points of the form $1/b$ are either $S$-numbers or $T$-numbers. Such generating functions are more general than those associated with automatic sequences.

Let $f(z)$ be an arbitrary $M$-function and let $\alpha$ be an arbitrary non-zero algebraic numbers such that $f(\alpha)$ is well-defined. 
Furthermore, assume that $f(\alpha)$ is transcendental. For every positive integer $d$, we can infer from Case (M$_q$) of Theorem~\ref{thm:measure_sans_system} with $r=1$ that there exist two positive real numbers, $A$ and $B$, such that the inequality 
$$
\vert P(f(\alpha))\vert >  A H(P)^{B} \,,
$$
holds for all non-zero polynomials $P \in \mathbb{Z}[X]$ of degree $d$. Thus, we have $w_d(f(\alpha))<\infty$. Since algebraic numbers are not Liouville numbers, taking $d=1$, we obtain that no element in $\mathbf M$ is a Liouville number, proving the second part of Corollary~\ref{coro:Lnumbers} and, more generally, that all transcendental numbers in $\bf M$ are either $S$-numbers or $T$-numbers in Mahler's classification. Furthermore, in the case where $\tau = 1$ in Theorem~\ref{thm:measure_sans_system}, we conclude that $f(\alpha)$ is an $S$-number.  This last result slightly generalizes that of Galochkin mentioned above.


\end{document}